\documentclass[a4paper, 12pt]{article}
\usepackage[plain]{fullpage}
\usepackage[utf8]{inputenc}
\usepackage[bbgreekl]{mathbbol}
\usepackage{bbm}
\usepackage{amsfonts}
\usepackage{amsmath,amscd}
\usepackage{eucal}
\usepackage{amssymb}
\usepackage{amsthm}
\usepackage{mathrsfs}
\usepackage{verbatim}
\usepackage{color}
\usepackage{faktor} 

\usepackage[all,cmtip]{xy}

\theoremstyle{plain}
\newtheorem{proposition}{Proposition}[section] 
\newtheorem{theorem}{Theorem}
\newtheorem{lemma}[proposition]{Lemma}
\newtheorem{corollary}[proposition]{Corollary}

\theoremstyle{definition}

\newtheorem{definition}[proposition]{Definition}
\newtheorem*{axiom}{Axiom}
\newtheorem{remark}[proposition]{Remark}

\newcommand{\R}{\mathbb R}
\newcommand{\Q}{\mathbb Q}
\newcommand{\N}{\mathbb N}

\renewcommand{\P}{\mathbb P}

\newcommand{\bcirc}{\bar\circ}
\newcommand{\hcirc}{\hat\circ}

\newcommand{\id}{\mathrm{id}}
\newcommand{\Id}{\mathrm{Id}}
\newcommand{\V}{\mathcal V}

\newcommand{\D}{\EuScript D}
\newcommand{\C}{\mathscr C}

\newcommand{\M}{\EuScript M}

\renewcommand{\k}{\Bbbk}

\newcommand{\E}{\EuScript E}
\newcommand{\F}{\EuScript F}
\newcommand{\bbF}{\mathbb{F}}
\newcommand{\bbS}{\mathbb{S}}

\newcommand{\vp}{\varphi}

\newcommand{\Cat}{{\EuScript{C} at}}

\newcommand{\Nat}{{\mathrm{Nat}}}

\newcommand{\Top}{{\EuScript{T}op}}

\newcommand{\CG}{\EuScript{CG}}

\newcommand{\mor}{\mathrm{mor}}

\newcommand{\Hom}{\mathrm{Hom}}

\newcommand{\Smash}{\wedge}
\newcommand{\ev}{\mathrm{ev}}

\newcommand{\op}{\mathrm{op}}

\newcommand{\cof}{\mathrm{cof}}

\newcommand{\tens}{\otimes}

\renewcommand{\S}{\sSet}

\newcommand{\colim}{\mathrm{colim}}

\newcommand{\Mod}{{\mathsf{Mod}}}

\newcommand{\coeq}{\mathrm{coeq}}
\newcommand{\map}{\mathrm{map}}

\renewcommand{\smash}{\Smash}

\newcommand{\circs}{\circ_\Sigma}
\renewcommand{\S}{\mathbb{S}}

\newcommand{\CMod}{{{_C}\Mod}}
\newcommand{\circr}{\circ_R}
\renewcommand{\F}{\mathbb{F}}
\newcommand{\FR}{{\mathbb F}_R}
\newcommand{\Thick}{\mathrm{Thick}}
\renewcommand{\E}{\EuScript{E}}
\newcommand{\G}{\EuScript{G}}
\newcommand{\homr}{\hom^r}
\newcommand{\homl}{\hom^l}
\newcommand{\Precat}{\mathsf{PreCat}}
\newcommand{\Prefun}{\mathsf{PreFun}}

\newcommand{\Bimod}{\mathsf{Bimod}}
\renewcommand{\Cat}{\mathsf{Cat}}
\newcommand{\Alg}{\mathsf{Alg}}
\renewcommand{\Top}{\mathsf{Top}}

\newcommand{\Algmod}{\Alg\Mod}
\newcommand{\Algbimod}{\Alg\Bimod}
\newcommand{\I}{\EuScript I}
\newcommand{\J}{\EuScript J}
\newcommand{\prop}{{\mathrm{prop}\,}}

\newcommand{\inj}{{ inj}}
\renewcommand{\smash}{\Smash}

\title{Koszul duality for categories with a fixed object set}
\author{Hadrien Espic}
\date{}

\newcommand{\Address}{{
  \bigskip
  \footnotesize
  
\textsc{Department of Mathematics, Stockholm University, SE-106 91 Stockholm, Sweden}\par\nopagebreak
  \textit{E-mail address:} \texttt{hadrien.espic@gmail.com}\par\nopagebreak
}}

\begin{document}

\maketitle

\begin{abstract}

We define a notion of Koszul dual of a monoid object in a monoidal biclosed model category. Our construction generalizes the classic Yoneda algebra $\mathrm{Ext}_A(\k, \k)$. We apply this general construction to define the Koszul dual of a category enriched over spectra or chain complexes. This example relies on the classical observation that enriched categories are monoid objects in a category of enriched graphs. We observe that the category of enriched graphs is biclosed, meaning that it comes with both left and right internal hom objects. Given a category $R$ (which plays the role of the ground field $\k$ in classical algebra), and an augmented $R$-algebra $C$, we define the Koszul dual of $C$, denoted $K(C)$, as the $R$-algebra of derived endomorphisms of $R$ in the category of right $C$-modules.

We establish the expected adjunctions between the categories of modules over $C$ and modules over $K(C)$. We investigate the question of when the map from $C$ to its double dual $K(K(C))$ is an equivalence. We also show that Koszul duality of operads can be understood as a special case of Koszul duality of categories. In this way we incorporate Koszul duality of operads in a wider context, and possibly clarify some aspects of it.

\end{abstract}

\tableofcontents

\section*{Introduction}





In this paper we develop and study a notion of Koszul duality for enriched categories with a fixed object set $S$, by way of a generalization of the Yoneda algebra construction to that setting. Let $C$ be a spectrally enriched category. Assuming $C$ is augmented, in a sense to be made precise below, we associate to $C$ another spectrally enriched category $K(C)$, with the same object set as $C$. The category $K(C)$ plays the role of the (derived) Koszul dual of $C$. We show that Koszul duality of operads (in spectra or in chain complexes) can be understood as a special case of our duality of categories. We believe that this can be applied to the Orthogonal Calculus of functors developed by Michael Weiss in \cite{Weiss95}, in a similar way that the Koszul duality of operads appears in Goodwillie Calculus (see \cite{AroneChing} for example).

Our definition of the Koszul dual of an enriched category with object set $S$ starts from the following standard idea. For $(\V,\smash,1)$ a pointed cofibrantly generated symmetric monoidal model category, one can consider the category $\V^{S\times S}$ of $\V$-graphs with vertex set $S$, which we will call precategories with object set $S$. This category of graphs has a monoidal structure given by the composition product
$$
(X\circ Y)(i,j)=\bigvee_{k\in S} X(i,k)\smash Y(k,j)
$$
and the unit $I$ defined by $I(i,j)=*$ for $i\neq j$ and $I(i,i)=1$. We explain in section 2 how this gives $\V^{S\times S}$ a structure of cofibrantly generated monoidal model category, whose category of monoids is exactly that of $\V$-categories with object set $S$. This point of view on categories is known (dating back to MacLane \cite{MacLaneCategories}, II.7) and is discussed in section 6 of \cite{SchwedeEquivalences} for example. However, it is usually not emphasized (to the author's knowledge) that this category is biclosed, meaning the functors $X\circ -$ and $-\circ X$ have right adjoints $\homl(X,-)$ and $\homr(X,-)$. 

\begin{align*}
    \homr(X,Y)(i,j) = \prod_{k\in S} \map(X(j,k), Y(i,k)) \\
\homl(X,Y)(i,j) = \prod_{k\in S} \map(X(k,i), Y(k,j))
\end{align*}

The existence of these left and right hom objects is what makes a generalization of the Yoneda algebra possible. Contrast this with the case of operads, where the monoidal category of symmetric sequences only has right hom objects, but no left hom objects.

This then fits in the more general framework discussed in section 1. If $R$ is a monoid in an arbitrary cofibrantly generated monoidal biclosed model category $(\C,\tens,1)$ and $C$ is an $R$-algebra, we can define the Koszul dual of $C$ by
$$
K(C) = \R \hom_C(R,R)
$$
This object comes equipped with a natural multiplicative structure given by the composition of right $C$-endomorphisms.

Another approach that is popular (especially in the setting of $\infty$-categories, see section 5.2 of \cite{lurieHigherAlgebra}) for tackling Koszul duality in a general setting is taking the dual of the bar construction $B(R,C,R)$, or equivalently just saying that the bar construction is the Koszul dual coalgebra of $C$. It relies on the existence of a comultiplicative structure on the bar construction (which in general seems to only be an $A_\infty$-coalgebra structure), and makes obtaining a map from $C$ into its double dual a non-trivial problem. Our approach has the advantage that the $R$-algebra structure is straightforward and that an algebra map from $C$ into its double dual is easily obtained. However, our $K(C)$ has the drawback of neither being augmented nor being functorial, although it is both of these up to homotopy. Let us remark though that even though $K(C)$ is not augmented, it has enough structure for us to provide a well-defined notion of double dual $KK(C)$. We give some basic examples in section 2.4 and also mention what we believe could be an interesting application of our construction to the category of finite sets and injections $FI$.

In section 3 we investigate the connection of Koszul duality of categories to Koszul duality of operads. Thus we consider the case where $R$ is the symmetric groupoid, and $C$ is the prop associated to a $\V$-operad. Here $\V$ can be any symmetric monoidal stable model category, such as the category of chain complexes or spectra. Regarding spectra, for technical reasons we choose to work with the category of $\S$-modules from \cite{EKMM}, also known as EKMM spectra. But we believe that our results should hold with any of the highly structured model categories of spectra described, for example, in \cite{mandell}. We obtain the following result, which compares our construction with Michael Ching's coalgebra structure on $BP$ from \cite{Ching05}.
\begin{theorem}[Proposition \ref{prop:StrongCompatibility}]
If $P$ is a reduced and levelwise finite operad in $\S$-modules, then there is a zigzag of weak equivalences of $\Sigma$-algebras
$$
K(\prop P) \sim \prop((BP)^\sharp)^\op
$$
\end{theorem}
Here $(BP)^\sharp$ denotes the arity-wise Spanier-Whitehead dual of $BP$. The algebra structure on the target is induced by the cooperad structure on $BP$ and the fact that the functor $\prop$ is monoidal (Proposition \ref{prop:PropFunctorMonoidal}). Let us mention as well that our proof of this result is constructive, building explicit $\Sigma$-algebra maps.


This result means that the approach to Koszul duality of categories we consider in this paper is compatible with the Koszul duality of operads in spectra developed by Michael Ching in \cite{Ching05}.

Following that, we recall in section 4 the idea of prefunctor. This is similar in spirit to the idea of seeing categories with object set $S$ as precategories with object set $S$ with a monoid structure. Prefunctors are elements of $\V^S$, and one can see functors from $C$ to $\V$ as the prefunctors possessing a right $C$-module structure, with respect to an external tensor product with $\V^{S\times S}$. We then derive adjunctions and Quillen equivalences relating the functor categories $\mathcal F (C,\V)$ and $\mathcal F(C^{op},\V)$ with $\mathcal F(K(C),\V)$. One of these equivalences, it turns out, is a special case of the work of Brooke Shipley and Stefan Schwede in \cite{SchSh03}.


Finally, in section 5, we remark that the algebra map from a finitely built square zero extension $R\vee M$ (i.e. $M$ is an $R$-bimodule and the algebra structure comes from this bimodule structure) to its double dual is a weak equivalence, and through a description of the bar construction $B(R,R\vee M, R\vee M)$ (see Proposition \ref{Prop:SquareZeroBarDescription}) prove the following result.

\begin{theorem}[Corollary \ref{Corollary:DualTrivialIsFree}]
Let $M$ be an $R$-bimodule and $R\vee M$ the corresponding square zero extension. Then there is a weak equivalence of algebras
$$
\varphi: \bbF_R (\hom_R(\Sigma M, R)) \xrightarrow{\sim} K(R\vee M)
$$
from the free $R$-algebra generated by $hom_R(\Sigma M, R)$ to the Koszul dual of $R\vee M$.
\end{theorem}
\vspace{4mm}
\begin{center}
    \textbf{Acknowledgements}
\end{center}
The author would like to thank his PhD supervisor Gregory Arone for his continuous supervision, help and guidance, and numerous enlightening conversations, especially regarding the main results in sections 3 and 5 of this paper. Special thanks go to Alexander Berglund as well for conversations about how the framework developed in this article relates to other discussions of Koszul duality in the literature and how it could be used in later work. The author would also like to thank Asaf Horev for his useful and interesting remarks on Koszul duality of algebras in the context of $\infty$-categories.

\section{A general Koszul dual construction}

We explain in this section how one can extend the notion of Koszul dual of an augmented algebra in a general context. Throughout this section, $(\C, \tens, 1)$ denotes a monoidal biclosed model category, not necessarily symmetric. In fact in further sections, $\C$ will not be symmetric. Besides we ask that the model structure on $\C$ is cofibrantly generated, so that we can transfer its model structure to the various categories of modules and algebras that will be considered. As stated in the introduction, the idea of our definition is to generalize the classical Yoneda algebra $\mathrm{Ext}_A(\k,\k)$ for $A$ a $\k$-algebra where $\k$ is a field. This first section does not really provide any new results. There is significant overlap with the first section of \cite{BerglundHess} (the arxiv version, even though it mostly deals with symmetric monoidal categories) and other similar discussions of model categories of algebras and modules in the literature. The role of this section is to lay out a very general context in which the phrase ``the derived $C$-endomorphisms of $R$" makes sense for $R$ a monoid and $C$ an augmented $R$-algebra, and then define the Koszul dual of $C$ to be just that.

\subsection{Basic definitions}

Let us start by recalling what we mean when we speak of a \textit{monoidal biclosed model category}.

\begin{definition}
A monoidal category $(\C, \tens, 1)$ equipped with a model structure is said to be \textit{biclosed} when for any object $A$ of it, the two functors $A\tens -$ and $-\tens A$ have right adjoints.

Moreover, it is called a \textit{monoidal model category} when the monoidal structure and the model structure satisfy the following compatibility condition, commonly referred to as SM7.

\begin{axiom}[SM7]
Assume that $f:A\rightarrow B$ and $g:X\rightarrow Y$ are cofibrations in $\C$. Then the induced map
$$
A\tens Y \cup_{A\tens X} B\tens X \rightarrow A\tens Y
$$
is a cofibration in $\C$, and is trivial if either $f$ or $j$ is. In addition, denoting $f:1_c \xrightarrow{\sim} 1$ the cofibrant replacement of the unit, the map $f\tens \id_A$ is a weak equivalence for all cofibrant $A$ in $\C$.
\end{axiom}
\end{definition}

The two right adjoint functors to $A\tens -$ and $-\tens A$ will be respectively denoted $\homl(A,-)$ and $\homr(A,-)$. We will often refer to the counits of the corresponding adjunctions as evaluation maps. Besides we will call
$$
\homl(-,-), \homr(-,-): \C^\op \times \C \rightarrow \C
$$
the left and right {\it hom bifunctors}.

This notion of monoidal model category comes from \cite{schwede-shipley}, Definition 3.1. For more context on this axiom, the reader can have a look at part II.3 of \cite{goerss-jardine}, even though it deals with the case of a model category equipped with a hom object functor landing in the category of simplicial sets instead of $\C$ itself.

\begin{remark}
We will simplify the notation by writing $\hom$ for $\homr$, since it is the hom object we will be using the most. The two hom object bifunctors are adjoint to each other in the contravariant variable, that is, for $A\in \C$,
$$
\mor_\C(\bullet, \hom(-, A)) \cong \mor_\C(\bullet \tens -, A) \cong \mor_\C(-, \homl(\bullet, A))
$$
\end{remark}

\begin{remark}\label{remark:EnrichedAdjunction}
Let $(F,G)$ be a pair of adjoint functors $\C\rightarrow \C$, and assume we want to know whether this pair is $\C$-enriched, meaning, if the adjunction bijection actually comes from an isomorphism of hom objects. Given $M,N$ in $\C$, there is a natural isomorphism $\mor(-, \hom(M,N)) \cong \mor(-\tens M, N)$ of functors from $\C$ to the category of sets and thus by the Yoneda lemma, an isomorphism $\hom(F(M), N) \cong \hom(M,G(N))$ is the same as an isomorphism $\mor(-\tens F(M), N) \cong \mor(- \tens M, G(N))$. So in the particular case where $F(X \tens M) \cong X\tens F(M)$ naturally for all $M,X$ or if, equivalently, $G(\homl(X,N))=\homl(X, G(N))$ naturally for all $X,N$, the adjoint pair is $\C$-enriched. In particular this applies to the case $F=-\tens Y$, $G= \hom(Y,-)$. Note that in general one can only expect this condition $F(X \tens M) \cong X\tens F(M)$ to hold in very particular cases, as this means the function $F$ is determined by the image of the unit object, but it will come in handy later.
\end{remark}

Given a monoid $R$ in $(\C, \tens, 1)$, we get categories of left and right $R$-modules which will be denoted $_R\Mod$ and $\Mod_R$ respectively, and a category of $R$-bimodules $_R\Mod_R := {_R(\Mod_R)} = (_R\Mod)_R$. As in section 4 of \cite{schwede-shipley}, given a monoid $R$ in $\C$, we define a tensor product over $R$
$$
-\tens_R - : {_S\Mod_R} \times {_R\Mod_T} \rightarrow {_S\Mod_T}
$$
as follows. For $A$ an $(S,R)$-bimodule and $B$ an $(R,T)$-bimodule, $A\tens_R B$ denotes the coequalizer of $A\tens R\tens B \rightrightarrows A\tens B$ where one map uses the right $R$-module structure on $A$ and the other uses the left $R$-module structure on $B$. In addition to this equivariant tensor product, one can also define an equivariant hom object functor for right $R$-modules as the equalizer of $\hom(A,B)\rightrightarrows \hom(A\tens R, B)$, denoted $\hom_R(A,B)$. Here one of the maps is $\hom(\mu_A, B)$, with $\mu_A$ being the right $R$-module structure map of $A$, and the other is $\hom(A\tens R, \mu_B) \circ \vp$ where by $\vp$ we mean the adjoint map to
$$
\ev\tens R : \hom(A,B) \tens A\tens R\rightarrow B\tens R
$$
Dually, one can define a relative hom object for left $R$-modules using $\homl$. Similarly as for the relative tensor product discussed above, it is important to remark that this relative hom object behaves well with respect to bimodules.

\begin{proposition}
Let $M$ be a $(T,R)$-bimodule and $N$ be an $(S,R)$-bimodule. Then $\hom_R(M,N)$ is an $(S,T)$-bimodule. In other words the $hom_R$ bifunctor admits a factorization
$$
\xymatrix{
& \C \\
{_T\Mod_R}^{op} \times {_S\Mod_R} \ar[ru]^-{\hom_R} \ar@{-->}[r] &{_S\Mod_T} \ar[u]_{U}
}
$$
\end{proposition}

\begin{proof}
Consider
$$
S\tens \hom_R(M,N) \tens M \rightarrow S\tens N \rightarrow N
$$
It is clear that this map yields an adjoint map
$$
S\tens \hom_R(M,N) \rightarrow \hom_R(M,N)
$$
since the left $S$-module and right $R$-module structures on $N$ are independent, i.e. they commute. Besides, this map is compatible with the monoid structure on $S$ because $S\tens N\rightarrow N$ is.
In a similar way, the composite
$$
\hom_R(M,N)\tens T \tens M\rightarrow \hom_R(M,N)\tens M \rightarrow N
$$
corresponds to a map
$$
\hom_R(M,N)\tens T \rightarrow \hom_R(M,N)
$$
which gives the desired right $T$-module structure to $\hom_R(M,N)$.
\end{proof}

The relative hom object $\hom_R$ gives $\Mod_R$ an enriched category structure over $\C$ and a monoidal biclosed structure on $({_R\Mod_R},\tens_R, R)$. This can be deduced from the definitions and the following result.

\begin{proposition}\label{prop:BimodulesMonoidalBiclosed}
For $A\in {_S\Mod_R}$, $B\in {_R\Mod_T}$, $C\in {_S\Mod_T}$ there is a natural isomorphism
$$
\mor_{_S\Mod_T}(A\tens_R B, C) \cong \mor_{_S\Mod_R}(A, \hom_T(B,C))
$$
\end{proposition}

\begin{proof}
Forgetting about the $S$- and $T$-actions for now, a map $\vp: A\tens_R B \rightarrow C$ is the same as a map $\vp: A\tens B\rightarrow C$ such that the following square diagram commutes.
$$
\xymatrix{
A\tens R \tens B \ar[r]^-{A\tens \mu_B} \ar[d]^{\mu_A \tens B} &A\tens B \ar[d]^\vp \\
A\tens B\ar[r]^-{\vp} &C
}
$$
Using the $(-\tens B,\hom(B,-))$ adjunction, this is equivalent to a map $\Bar{\vp}$ with a commuting diagram
$$
\xymatrix{
A\tens R \ar[rd]^-{g} \ar[d]^{\mu_A}&\\
A \ar[r]^-{\Bar{\vp}} &\hom(B,C)
}
$$
where $g$ is the adjoint map to $\vp\circ (A\tens \mu_B)$. Note that $g$ is also adjoint to the map $(\mu_B)^* \circ \Bar{\vp}: A\rightarrow \hom(R\tens B,C)$, so the latter diagram means exactly that $\Bar{\vp}$ is a map of right $R$-modules.

Secondly, a map $\vp : A\tens B \rightarrow C$ is a map of right $T$-modules if and only if
$$
\xymatrix{
A\tens B\tens T \ar[r]^-{A\tens \mu_B} \ar[d]^{\vp\tens T} & A\tens B \ar[d]^\vp \\
C\tens T \ar[r]^-{\mu_C} &C
}
$$
commutes, which is the case if and only if
$$
\xymatrix{
A \ar[rd]^-{(\mu_B)^*} \ar[d]^{\overline{\vp\tens T}}&\\
\hom(B\tens T, C\tens T) \ar[r]^-{(\mu_C)^*} &\hom(B\tens T,C)
}
$$
does, which is equivalent to the property that the adjoint map $\Bar{\vp}$ to $\vp$ factors through the equalizer $\hom_T(B,C)$.

In addition, using a similar argument, $\vp: A\tens B\rightarrow C$ is a left $S$-module map if and only if $\Bar{\vp}: A\rightarrow \hom(B,C)$ is a left $S$-module map.
\end{proof}

A monoid in $({_R\Mod_R}, \tens_R, R)$ is called an $R$-algebra, and whenever we speak of a module over an $R$-algebra it will be with respect to this monoidal biclosed category of $R$-bimodules.

\subsection{Transfer of model structures}

Many people, going back to Kan, have discussed the conditions under which one can transfer a cofibrantly generated model structure along an adjoint pair of functors. The specific problem of transfer of a model structure from a monoidal model category to the category of its monoids and the category of modules over a given monoid is explored in \cite{schwede-shipley} in the symmetric monoidal case, and in \cite{hovey2} and \cite{batanin-berger} in a more general context. These last two sources are the ones we will be referring to here. 

Given a cofibrantly generated monoidal biclosed category $\C$ like in the previous section, there are two important assumptions we want a model structure on it to satisfy. First, that the \textit{monoid axiom} holds in $\C$. This axiom is due to \cite{schwede-shipley} (Definition 3.3) and states that any transfinite composition of pushouts of maps of the form $X\tens f$ or $f \tens X$, where $f$ is a trivial cofibration, is a weak equivalence. As stated in Lemma 3.5 of that same source, it is enough to check it for $f$ a generating trivial cofibration.

\begin{proposition}[Theorem 2.1 in \cite{hovey2}]\label{prop:HoveyTransfer}
Let $\C$ be a cofibrantly generated monoidal model category with generating cofibrations $I$ and generating trivial cofibrations $J$. Let $A$ be a monoid in $\C$ and assume that the domains of $I$ (resp.$J$) are small relative to $A\tens I$-cell (resp. $A\tens J$-cell). Suppose moreover that every map of $A\tens J$-cell is a weak equivalence in $\C$, which is a condition that can be seen as a monoid axiom specific to $A$ (this holds trivially when $A$ is cofibrant). Then, the category $_A\Mod$ inherits a cofibrantly generated model structure through $A\tens -$.
\end{proposition}

In fact, corollary 2.2 in \cite{hovey2} tells us that when $A$ is cofibrant, not only does $_A\Mod$ inherit a projective model structure for free, but the forgetful functor $U: _A\Mod \rightarrow \C$ creates cofibrations as well. Besides, we will assume that the category $\C$ is \textit{compactly generated} as in Definition 2.4 of \cite{batanin-berger} recalled below (see section 2 of that paper for a more thorough discussion).

\begin{definition}\label{def:CompactlyGenerated}
A class $K$ of morphisms in $\C$ is said to be {\it saturated} if it is closed under cobase change, transfinite composition and retracts. For $K$ saturated, we say the class $W$ of weak equivalences in $\C$ is {\it $K$-perfect} if $W$ is closed under filtered colimits along morphisms in $K$.

A monoidal model category $\C$ is said to be {\it $K$-compactly generated} if
\begin{enumerate}
\item it is cofibrantly generated,
\item the class of weak equivalences of $\C$ is $K$-perfect, and
\item every object of $\C$ is small relative to $K$.
\end{enumerate}
Moreover $\C$ is simply called {\it compactly generated} when it is $I^\tens$-compactly generated, where $I^\tens$ is the monoidal saturation of $I$, that is, the closure of the class of cofibrations under tensoring with arbitrary objects of $\C$.
\end{definition}

The category of compactly generated weak Hausdorff spaces, and suitable categories of simplicial sets or spectra are for example compactly generated and satisfy the monoid axiom; see section 2 for further discussion on this.

\begin{proposition}[Theorem 4.1 in \cite{batanin-berger}]\label{prop:BataninBergerTransfer}
If a compactly generated monoidal model category $\C$ satisfies the monoid axiom, the category of monoids in $\C$ inherits a transfered model structure through the adjunction given by the free monoid construction and the forgetful functor.

In addition, if $\D$ is a small $\C$-enriched category, then we get a projective model structure on $[\D,\C]$. And if $\D$ possesses a symmetric monoidal structure, the Day convolution product induces a compactly generated monoidal model category structure on $[\D,\C]$ which satisfies the monoid axiom.
\end{proposition}

We can use the two previous results to get the model structures we want. By the first statement of \ref{prop:BataninBergerTransfer}, we obtain a model structure on the category $\mathrm{Mon}\C$ of monoids in $\C$ which is cofibrantly generated with $I_{\mathrm{Mon}\C}= F(I)$ and $J_{\mathrm{Mon}\C}= F(J)$ where $F$ is the free monoid functor. On the other hand, we can use the second part of Proposition \ref{prop:BataninBergerTransfer} with $\D= R^{op} \times R$ to get a projective model structure on $_R\Mod_R$.

\begin{proposition}\label{prop:BimodulesMonoidalCategory}
Assuming $R$ is cofibrant in $\C$, the category of $R$-bimodules \linebreak $(_R \Mod _R, \tens_R, R)$ is a monoidal model category.
\end{proposition}

\begin{proof}
We have explained above how this is a monoidal biclosed structure. As mentioned in Lemma 3.5 of \cite{schwede-shipley}, the pushout-product axiom only needs to be checked for generating cofibrations and trivial cofibrations, which means it is enough to check it for $R\tens f \tens R$ and $R\tens g \tens R$ where $f:A\rightarrow B$, $g:X\rightarrow Y$ are generating cofibrations in $\C$. The pushout of $R\tens f \tens R \tens_R R \tens X \tens R$ and $R\tens A \tens R \tens_R R \tens g \tens R$ is isomorphic to the pushout of $R\tens f \tens R\tens X \tens R$ and $R\tens A \tens R\tens g \tens R$. Because $R$ is cofibrant in $\C$, tensoring with $R$ on one side or the other is left Quillen i.e. preserves cofibrations and trivial cofibrations, which means we can simply use the pushout-product axiom in $\C$ with $R\tens f \tens R$ and $g\tens R$.

We also need to prove the extra condition about the unit, namely that for any $R$-bimodule $A$ that is $R$-cofibrant (meaning, cofibrant as an $R$-bimodule) the map $R_c\tens_R A \rightarrow R\tens_R A$ is a weak equivalence, where $R_c$ is an $R$-cofibrant replacement of $R$. The transferred model structure on ${_R}\Mod_R$ is such that $A$ is $R$-cofibrant if and only if the map $*\rightarrow A$ is a retract of transfinite compositions of pushouts of maps like $R\tens f \tens R$ where $f$ is a generating cofibration of $\C$. Therefore it is enough to consider $A=R\tens A' \tens R$ where $A'$ is a cofibrant object of $\C$. In addition, since $R$ is cofibrant in $\C$, any $R$-cofibrant object is cofibrant in $\C$ (this is because every map of the form $R\tens f \tens R$ as above is a cofibration in $\C$); in particular this is the case for $R_c$. Therefore the map $R_c\rightarrow R$ is a weak equivalence of cofibrant objects in $\C$, meaning that by Ken Brown's lemma (Lemma 1.1.12 in \cite{hovey}), it is sent to a weak equivalence by the functor $-\tens A' \tens R$, hence the map
$$
R_c \tens_R A \rightarrow R\tens_R A
$$
is one as well. 
\end{proof}

It is then immediate that we can transfer this model structure on $R$-bimodules to categories of $R$-algebras and of $C$-modules for $C\in \Alg_R$. Indeed, assuming $C$ is cofibrant as an $R$-bimodule, which as we will see later is a reasonable assumption to make, we can then use \ref{prop:HoveyTransfer} to get a projective model structure on $\Mod_C$. Moreover, for any generating cofibration $R\tens f \tens R$ of $_R\Mod_R$, there is a canonical isomorphism $R\tens f \tens R \tens_R M \cong R\tens f \tens M$. Because $\tens$ distributes over colimits, in particular it distributes over transfinite composition of pushouts. Therefore, if the monoid axiom holds in $\C$ it holds in $_R\Mod_R$ as well, implying that we can transfer its model structure to the category $\Alg_R$ of $R$-algebras.

\subsection{The Koszul dual of an augmented $R$-algebra}

In the classical construction of the Koszul dual of an algebra, one starts by fixing a ground field $\mathbb k$ and an augmented algebra $A\rightarrow \mathbb k$. It is useful to keep in mind the analogies between this and the general context we're interested in. In the remainder of this section, $R$ is a cofibrant monoid in $\C$ and $C$ is an $R$-algebra equipped with an augmentation $\epsilon : C\rightarrow R$ ($\epsilon$ is a map of monoids). Notice that the map
$$
R\tens_R C \longrightarrow R\tens_R R \longrightarrow R
$$
obtained from $\epsilon$ and the multiplicative structure of $R$ gives $R$ a right $C$-module structure (this map is essentially the same as $\epsilon$).

\begin{definition}\label{def:KoszulDual}
The Koszul dual of $C$, denoted $K(C)$, is the derived endomorphism object of $R$ in the category of right $C$-modules $\hom_C( R_{cf}, R_{cf})$, where $R_{cf}$ denotes a cofibrant fibrant approximation of $R$ in $\Mod_C$. 
\end{definition}

As we will see below, it is enough to ask $R_{cf}$ to be a cofibrant fibrant approximation and not a replacement, that is, we do not the weak equivalence $ R_{cf} \xrightarrow{\sim} R$ to be a fibration. Note that alternatively, the Koszul dual could be defined as the derived endomorphism object of $R$ in left $C$-modules, using the other homomorphism object $\homl$.


The Koszul dual is naturally equipped with an $R$-algebra structure with the composition of endomorphisms as multiplication. This kind of argument is classical but it is useful to provide the reader with a careful proof of it in our setting.

\begin{proposition}
Let $M$ be a right $C$-module. The endomorphism object $\hom_C(M,M)$ has a natural $R$-algebra structure consisting in the composition of endomorphisms.
\end{proposition}

\begin{proof}
Recall that $(_R\Mod_R, \tens_R,R, \hom_R(-,-))$ was proven to be a monoidal biclosed category in Proposition \ref{prop:BimodulesMonoidalCategory}. Let $M$ be a right $C$-module. First, notice that $\hom_R(M,M)$ is naturally equipped with a monoid structure because $_R\Mod_R$ is enriched over itself. Explicitly, the multiplication is given by the composition map defined as the adjoint of
$$
\ev_{M,M} \circ (\id \tens_R \ev_{M,M}) : \hom_R(M,M) \tens_R \hom_R(M,M) \tens_R M \rightarrow M
$$
To see that the subobject $\hom_C(M,M)$ is stable under composition, we show that
$$
\hom_C(M,M)\tens_R \hom_C(M,M) \tens_R M \rightarrow M
$$
is a right $C$-module map. It is enough to show that $\ev_{M,M}$ restricted to $\hom_C(M,M)\tens M$ is one, which is true since the following diagram commutes by definition of $\hom_A(M,M)$.
$$
\xymatrix{
\hom_A(M,M)\tens M \ar[d]^{\id \tens \mu_M} \ar[r] & M\tens A \ar[d]^{\mu_M} \\
\hom_A(M,M)\tens M \ar[r] & M
}
$$
\end{proof}

The reader might object that the Koszul dual construction from Definition \ref{def:KoszulDual} was only defined up to a choice of cofibrant fibrant approximation of $R$ in $\Mod_C$. And even though it is easy to see that two different choices yield two definitions of $K(C)$ that are equivalent as $R$-bimodules, we need to know that they are equivalent as $R$-algebras.

\begin{proposition}\label{prop:TylerLawsonThing}
Given $\bar R$  and $\hat{R}$ two fibrant cofibrant approximations of $R$ as a right $C$-module, there is a zigzag of weak equivalences of $R$-algebras
$$
\hom_C(\bar R, \bar R) \xleftarrow{\sim} \bullet\xrightarrow{\sim} \hom_C(\hat{R}, \hat{R})
$$
\end{proposition}

\begin{proof}
Without loss of generality, we can assume that $\hat R$ is actually a cofibrant {\it replacement} of $R$ in $\Mod_C$.

The idea is to use an endomorphism object for the category of arrows in $\Mod_C$ as the middle object. Mostly, we will simply apply the results of section 2 of \cite{davis2014commutative}. In the proof of Lemma 2.3, Davis and Lawson show that if $x\xrightarrow{\sim} y$ is a weak equivalence between cofibrant fibrant objects in a model category $\M$, there exists a diagram of acyclic fibrations
$x \twoheadleftarrow \tilde{x} \twoheadrightarrow y$.

Now, coming back to our setting, there is a lift
$$
\xymatrix{
& \hat R \ar@{->>}[d]^\sim\\
\bar R \ar@{-->}[ru]^-\sim \ar[r]^-\sim &R
}
$$
Therefore we can use this lemma with $x=\bar R$ and $y=\hat R$ to get a zigzag of acyclic fibrations between $\bar R$ and $\hat R$. We can then conclude by applying Corollary 2.8 in the same source, stating that for any acyclic fibration between fibrant objects, there is a zigzag of weak equivalences of algebras between the endomorphisms of its source and those of its target.
\end{proof}


\begin{definition}
A more general construction can be considered by assigning to any pair $(C, M)$ where $C$ is an $R$-algebra and $M$ a right $C$-module the object
$$
K(C,M)= \R\homr(M,M)
$$
and similarly if $M$ is a {\it left} $C$-module with $\R\homl$ instead. Then, the Koszul dual of an augmented algebra $C$ is the particular case $K(C)=K(C,R)$.
\end{definition}

Note that the evaluation map equips $R_{cf}$ with a structure of left $K(C)$-module. we can consider the object
$$
K(K(C)) := K(K(C), R_{cf}) = \homl_{\mathrm{left }K(C)} (\tilde{R_{cf}},\tilde{R_{cf}})
$$
of derived endomorphisms of $R_{cf}$, where $\tilde{(-)}$ denotes a chosen cofibrant fibrant approximation functor in the category of left $K(C)$-modules. We wrote $\homl_{\mathrm{left }K(C)}$ here to emphasize the fact that we are considering equivariant homomorphism objects of left $K(C)$-modules, but as explained in the beginning of section 1.1 this notation is redundant since $\homl$ can only be used for equivariant left module homomorphisms and $\hom=\homr$ can similarly only be used for equivariant right module homomorphisms. Thus we will just write $\hom_{K(C)}^l$ instead.

A useful perk of opting for the derived endomorphisms of $R$ approach for the Koszul dual $R$-algebra of $C$ is that, in addition to the multiplicative law being given by endomorphism composition, a natural $R$-algebra map from $C$ into its double dual is easily obtained. There is a natural map
$$
\nu: C\xrightarrow{\phi} \homl_{K(C)}(R_{cf},R_{cf})\rightarrow K(K(C))
$$
where $\phi$ is defined as the adjoint map to the right $C$-module structure map $\hat{\phi}: R_{cf} \tens_R C \rightarrow R_{cf}$ on $R_{cf}$.

\begin{proposition}
The map $\phi$ is a well-defined algebra map. Moreover if $\tilde{(-)}$ is an $_R\Mod_R$-enriched functor (meaning it satisfies $\homl_{K(C)}$-functoriality conditions), then $\nu$ is an algebra map as well.
\end{proposition}

\begin{proof}
We will only prove the multiplicativity of these maps, since preservation of the unit is clear in both cases. First note that if $\phi$ is well defined, then it is automatically an algebra map because $\hat{\phi}$ satisfies associativity and unitality conditions, being a $C$-module structure map.

Showing that $\phi$ is well-defined consists in justifying the factorisation
$$
\xymatrix{
C\ar[rd]^-{\phi} \ar@{-->}[r] &\homl_{K(C)}(R_{cf}, R_{cf}) \ar[d] \\
&\homl(R_{cf}, R_{cf})
}	
$$
As mentioned in the proof of \ref{prop:BimodulesMonoidalBiclosed}, such a factorization exists if and only if the adjoint map $\hat{\phi}$ is a left $K(C)$-module map, that is, if and only if the following square diagram commutes.
$$
\xymatrix{
\hom_C(R_{cf}, R_{cf}) \tens_R R_{cf} \tens_R C \ar[d] \ar[r] &\hom_C(R_{cf}, R_{cf}) \tens_R R_{cf} \ar[d] \\
R_{cf} \tens_R C \ar[r] &R_{cf}
}
$$
This holds if and only if the evaluation map involving $\hom_C$ is a map of right $C$-modules. Note that the three equivalent properties we just stated are just expressing the fact that $R_{cf}$ is a $(K(C), C)$-bimodule. We can conclude by observing that since the inclusion $\hom_C(R_{cf},R_{cf}) \hookrightarrow \hom(R_{cf},R_{cf})$ obviously factors through $\hom_C(R_{cf},R_{cf})$, its adjoint map $\mathrm{ev} : \hom_C(R_{cf},R_{cf})\tens R_{cf} \rightarrow R_{cf}$ is a right $C$-module map.

Regarding the second statement, assuming that $\tilde{(-)}$ is $\homl_{K(C)}$-functorial, then in particular $\nu$ is multiplicative, since the composition law on the two endomorphism objects $\homl_{K(C)}(R_{cf}, R_{cf})$ and $K(K(C))$ is  simply the composition of endomorphisms.
\end{proof}

\begin{remark}
The condition of $\homl_{K(C)}$-functoriality we ask from the cofibrant-fibrant replacement functor is reasonable for the following reason. In the general setting of a monoid $A$ in a monoidal biclosed model category, an endofunctor in the category of left $A$-modules is $\homl$-enriched in the special case where there is a natural isomorphism $F(M\tens X)\cong F(M)\tens X$ for any $X$ (not necessarily an $A$-module), because in this case a map
$$
\homl_A(M,N) \rightarrow \homl_A(F(M),F(N))
$$
is provided by the adjoint to the application of $F$ to the evaluation map \linebreak $M\tens \hom_A(M,N)\rightarrow N$. This condition is satisfied by the bar construction $B(A,A,-)$ (discussed below) which will be our preferred choice of functor in what follows - in the context of left modules, whereas it will be $B(-,A,A)$ for right $A$-modules.
\end{remark}

The question of finding conditions under which the above defined map $\nu$ is a weak equivalence is interesting to explore (and will be in section 5 to some extent) but for now we will shift our focus to describing a classical example of a functor that gives a cofibrant fibrant approximation in a rather general case: the bar construction. We will make the following assumptions to simplify the notation and make our life easier.

{\bf Assumptions.} From now on, we assume that all objects in the monoidal model category $\C$ are fibrant. We consider a cofibrant monoid $R$ and an $R$-algebra $C$ that is $R$-cofibrant. We assume as well that the unit map $R\rightarrow C$ is a cofibration of $R$-bimodules.

\begin{definition}\label{def:BarConstruction}
Recall that for a right $C$-module $M$ and a left $C$-module $N$, the bar construction $B(M,C,N)$ is the geometric realization of the simplicial object $B_\bullet(M,C,N)$ (which is usually called the \textit{simplicial} bar construction) defined in simplicial degree $l$ by
$$
M\tens_R C^{\tens_R l} \tens_R N 
$$
The face maps are constructed using the right $C$-module structure on $M$, the multiplicative structure on $C$, and the left $C$-module structure on $N$ at the various $\tens_R$ symbols, while the degeneracy maps use the unit map $R\rightarrow C$ to increase the simplicial degree by one.
\end{definition}

This is a classical and extensively studied construction, but we have to be slightly careful here in explaining what we mean by the {\it geometric realization} of a simplicial object $X_\bullet$ in a general model category $\M$. The reader can turn to chapter 3 of \cite{fresse17} for a careful discussion on this topic involving the concept of Reedy frame.

We can make things a bit simpler by limiting ourselves to the case where $X_\bullet$ is Reedy cofibrant. Consider a Reedy cofibrant approximation $Z$ of the constant cosimplicial object $cc_*X$ in the category $\mathrm{cs}\M$ of cosimplicial simplicial objects in $\M$. One can then define the geometric realization of $X$ as the coend of $Z$.


\begin{remark}\label{remark:GeomReal}
This is a very general definition of geometric realization, with the disadvantage of course of not providing an explicit description without a given cofibrant replacement of $cc_*X$. We will not need to worry about this at all in later sections as any model category $\M$ we will consider comes equipped with a ``nice" external tensor product enabling us to make sense of the formula
$$
\vert X \vert = \int^{n\in \Delta} \Delta^n \smash X_n
$$
\end{remark}

\begin{proposition}
Let $R$ be a monoid in $\C$, and $C$ an $R$-cofibrant $R$-algebra. Given $M$ an $R$-cofibrant right $C$-module, the bar construction $B(M,C,C)$ is a cofibrant approximation of $M$ in $\Mod_C$.
\end{proposition}

\begin{proof}
The bar construction has a right $C$-module structure coming from the composition law on $C$, since the degreewise tensoring functor $-\circ_R C: s{_R\Mod_R}\rightarrow s{_R\Mod_R}$ commutes with geometric realization. The proof that the map
$$
B(M,C,C)\rightarrow M
$$
induced from the right $C$-module structure on $M$ provides a cofibrant approximation of $M$ is similar to the case of operads and modules discussed in section 9 of \cite{AroneChing} (see Lemma 9.18 and Proposition 9.21).
\end{proof}

As a particular case assume that all objects in $\C$ are fibrant, and hence that the same holds for $_R\Mod_R$ and $\Mod_C$ equipped with the projective model structure. Then $B(R,C,C)$ becomes a fibrant cofibrant approximation of $R$ in the category of right $C$-modules, which means $K(C)$ is the right $C$-module endomorphism object
$$
K(C) = \hom_C (B(R,C,C), B(R,C,C))
$$

\begin{corollary}
The map $B(R,C,C)\rightarrow R$ obtained from the augmentation $\epsilon : C\rightarrow R$ induces an equivalence of $R$-bimodules
$$
K(C)\simeq \hom_C(B(R,C,C),R)
$$
\end{corollary}

\begin{lemma}
There is an isomorphism of $R$-bimodules
$$
B(R,C,C)\tens_C R \cong  B(R,C,R)
$$
\end{lemma}

\begin{proof}
First, $B(R,C,C) \tens_R M = \vert B_\bullet (R,C,C) \vert \tens_R M = \vert B_\bullet (R,C,C) \tens_R M \vert$ for all $M$ since $- \tens_R M$ has a right adjoint. So
\begin{align*}
B(R,C,C)\tens_C R &= \coeq(B(R,C,C) \tens_R C \tens_R R \rightrightarrows B(R,C,C)\tens_R R) \\
&\cong \coeq(\vert B_\bullet(R,C,C) \tens_R C \tens_R R \vert \rightrightarrows \vert B_\bullet(R,C,C)\tens_R R\vert) \\
&\cong \vert \coeq(B_\bullet(R,C,C) \tens_R C \tens_R R \rightrightarrows B_\bullet(R,C,C)\tens_R R) \vert \\
&= \vert R\tens_R C^{\tens_R \bullet} \tens_C C \tens_R R \vert\\
&\cong \vert R \tens_R C^{\tens_R \bullet} \tens_R R\vert \\
&= B(R,C,R)
\end{align*}

\end{proof}

\begin{lemma}
Let $A$ be a monoid in a monoidal biclosed category $\C$. Let $M$ be a right $A$- module, $N$ a left $A$-module and $P$ any object of $C$. Then there is a natural isomorphism
$$
\hom_A(M,\hom(N,P)) \cong \hom(M\tens_A N, P)
$$
or in other words, the adjoint pair
$$
-\tens_A N : \Mod_A \rightleftarrows \C : \hom(N,-)
$$
is $\C$-enriched.
\end{lemma}

\begin{proof}
First, recall the classical observation that there is a natural isomorphism
$$
\hom(I\tens J, K)\cong \hom(I, \hom(J,K))
$$
for $I,J,K$ in $\C$ (see Remark \ref{remark:EnrichedAdjunction}). Now, by definition $\hom_A(M,\hom(N,P))$ is the equalizer of the diagram
$$
\hom(M,\hom(N,P)) \rightrightarrows \hom(M\tens A, \hom(N,P))
$$
which is isomorphic to the following diagram.
$$
\hom(M\tens N, P)) \rightrightarrows \hom(M\tens A \tens N,P)
$$
And since $\hom(-,P)$ sends colimits to limits, its equalizer is $\hom(M\tens_A N, P)$, which completes the proof.
\end{proof}

As a consequence of these lemmas, we get the following description of the Koszul dual construction's underlying $R$-bimodule.

\begin{proposition}\label{prop:DescriptionDual}
There is an equivalence of $R$-bimodules
$$
K(C) \simeq \hom_R(B(R,C,R), R)
$$
\end{proposition}

\begin{proof}
By the first lemma, showing that $K(C) \simeq \hom_R(B(R,C,C) \tens_C R, R)$ is enough. From the definition, we have $K(C) \simeq \hom_C(B(R,C,C), R)$, and we can then use the fact that $\hom_R(R,R) \cong R$ coupled with the second lemma to conclude.
\end{proof}

Of course we would ideally want a description of the Koszul dual that unlike this one, says something about its multiplicative structure. In section 3, we will investigate how to improve this result in one of the examples we are most interested in and get actual equivalences of {\it algebras}. For now, we have to stop here and instead present the special setting we want to study, where the base monoidal biclosed model category $\C$ is one whose monoids are enriched categories with a fixed object set.

\section{Categories as algebras in a monoidal category}

Let us fix a cofibrantly generated pointed symmetric monoidal model category $(\V, \Smash, 1)$, whose hom object bifunctor (recall that $\homl=\homr$ in the symmetric case) will be written $\map(-,-)$.  For example, $\V$ can be the category $\CG_*$ of pointed compactly generated weak Hausdorff spaces, or a suitable category of spectra like the category of $\bbS$-modules. In what follows we explain how one can see $\V$-categories with a fixed object set $S$ as algebras in some monoidal biclosed category $\Precat$ of ``$\V$-graphs", an idea that has been used a number of times in the literature. This section has significant overlap with \cite{SchwedeEquivalences}, section 6, although they do not mention the existence of hom objects. We are then able to use the framework discussed in section 1 to define the Koszul dual of an augmented $R$-algebra $C$, where $R$ is a category with object set $S$.

\subsection{Definitions}

Let us start by fixing a set $S$, which will be the vertices/object set for all things considered in these notes. In the examples we discuss, we have $S=\N$, since we want to consider an environment in which props of operads naturally live, but there is no need to restrict ourselves to this case for the general definitions. One might for example be interested in setting $S=<\mathscr C>$, that is, the free commutative monoid generated by a given set of colors $\mathscr C$, in order to include props of $\mathscr C$-colored operads.

\begin{definition}
A \textit{precategory} (with object set $S$) is the data, for all couples $(s,s')\in S$, of an object $E(s,s')$ in $\V$. In other words, a precategory is simply a $\V$-enriched directed graph. Note that we allow self-arrows and parallel arrows. The category of precategories is the functor category $\V^{S\times S}$, which will be denoted by $\Precat$.
\end{definition}

The idea behind the name ``precategory" is that a precategory is a category that is missing a multiplicative and unital structure. This idea can be made formal in the following way (\cite{MacLaneCategories}, II.7). Categories with object set $S$ are exactly precategories equipped with a monoid structure for the composition product defined by
$$
(C\circ D)(i,j) = \bigvee_{k\in S} C(i,k) \Smash D(k,j)
$$
for all $i,j$ in $S$. This gives us a (non-symmetric) monoidal category structure on $\Precat$, with the unit object $I$ defined by $I(s,s)=1$ for all $s$ and $I(s,s')= *$ for all $s\neq s'$. We will usually not bother adding a ``$\V$-" prefix in the rest of this paper, but the reader should keep in mind that all the precategories and categories we consider are $\V$-enriched. Similarly we will not write anything to specify the object set $S$.

This category has the following right and left hom-objects $\homr$ and $\homl$, making $\Precat$ biclosed.
\begin{align*}
    \homr(C,D)(i,j) = \prod_{k\in S} \map(C(j,k), D(i,k)) \\
\homl(C,D)(i,j) = \prod_{k\in S} \map(C(k,i), D(k,j))
\end{align*}

The notation is slightly unfortunate, but from now on, we will write $\Cat$ for the category of categories with object set $S$, or $\Cat_S$ when there is a risk for confusion. For $R$ in $\Cat$, the monoid structure relative to $\circ$ allows us to consider modules and algebras over $R$, and modules over an $R$-algebra $C$ just like we did in the more general context of section 1.

\begin{remark}\label{Remark:InvolutionPrecatOp}
Even though the category of precategories is not symmetric, it is equipped with an involution $(-)^{op}: \Precat \rightarrow \Precat$ defined by $M^{op}(i,j)= M(j,i)$ for all $i,j\in S$. This involution is anticommutative, that is, $(M\circ N)^{op} = N^{op} \circ M^{op}$. As a consequence, if $R$ is a category with object set $S$, the maps $I=I^{op} \rightarrow R^{op}$ and $R^{op}\circ R^{op}= (R\circ R)^{op}\rightarrow R^{op}$ give $R^{op}$ a structure of category with object set $S$ as well. Moreover, a left $R$-module structure on $M$ is the same as a right $R^{op}$-module structure on $M^{op}$. Another interesting property is that $\hom^r(A^{op}, B^{op})^{op}= \hom^l(A,B)$.
\end{remark}

We will write $_R\Mod_R$ for the category of $R$-bimodules. As explained in section 1, the composition product $\circ$ of $\Precat$ induces on $_R\Mod_R$ a composition product over $R$, denoted $\circ_R$ where $(M\circ_R N)(i,j)$ is the coequalizer of the diagram
$$
\xymatrix{
\bigvee_{k,l\in S} M(i,k) \Smash R(k,l) \Smash N(l,j) \ar@<-2pt>[r] \ar@<2pt>[r] & \bigvee_{s\in S} M(i,s) \Smash N(s,j)
}
$$
where one of the arrows comes from the right $R$-module structure on $M$ and the other one from the left $R$-module structure on $N$. Recall that in addition $-\circ_R M$ and $M\circ_R -$ have right adjoints constructed in section 1 as the $R$-equivariant hom object functors $\homr_R(M,-)$ and $\homl_R(M,-)$, making $(_R\Mod_R, \circ_R, R)$ into a monoidal biclosed model category.

\begin{remark}\label{remark:OperadsNotClosed1}
The fact that the category of $R$-bimodules is biclosed is crucial, and means it is in that respect better behaved than the category of symmetric sequences that appears in the theory of operads. Informally, the problem with defining homomorphism objects of symmetric sequences lies in the fact one should see them as $X(-,1)$ for some $\Sigma$-bimodule $X$, with $\Sigma$ denoting the symmetric groupoid. We will recall later the definition of prop of a symmetric sequence that provides such a choice of $X$ and prove that it is a monoidal functor from symmetric sequences to $\Sigma$-bimodules.
\end{remark}

\subsection{The free object functors}

Let us fix $R$, a category with object set $S$. The categories defined above are linked by forgetful functors, which we will always denote "U" regardless of the source and target categories. Those forgetful functors all have left adjoint functors, the free left module, free right module, free category, free algebra functors discussed below. Everything is summarized in the diagram below.

$$
\xymatrix{
\Precat \ar@<2pt>[r]^-{R\circ -} \ar@<2pt>[d]^-{\F} & _R\Mod \ar@<2pt>[l]^-{U} \ar@<2pt>[r]^-{-\circ R} &_R\Mod_R \ar@<2pt>[l]^-{U} \ar@<2pt>[d]^-{\FR} \\
\Cat \ar@<2pt>[u]^-{U} & & \Alg_R \ar@<2pt>[u]^-{U}
}
$$

For $E$ a precategory, the left module structure on $R\circ E$ is simply induced by the monoid structure on $R$, and any left $R$-module map $R\circ E \rightarrow M$ is uniquely determined by its restriction to $I\circ E =E$, meaning that we indeed have an adjoint pair $R\circ - : \Precat \leftrightarrows {_R\Mod} : U$. The case of $-\circ R: {_R\Mod} \leftrightarrows _R\Mod_R: U$ is similar. One just needs to notice that if $M$ is a left $R$-module, then $M\circ R$, in addition to having a right $R$-module structure induced by the monoid structure on $R$, inherits a left $R$-module structure from $M$. The fact that those two structures are compatible follows from associativity of $\circ$. Of course similar observations can be made for modules over an $R$-algebra $C$ using the equivariant composition product $\circ_R$.

The other two pairs in the diagram above are instances of a free monoid functor adjunction, the left-hand one for $(\Precat, \circ, I)$ and the right-hand one for $(\Bimod_R, \circ_R, R)$. Notice that we only need to discuss the latter, because the former corresponds to the special case $R=I$. We are going to use the following result.

\begin{proposition}\label{prop:freemonoid}
Let $(\M, \tens, 1)$ be a monoidal category. Suppose that $\M$ has countable coproducts and that $\tens$ distributes over countable coproducts. Then the functor
$$
\F(M) = \bigoplus_{n\ge 0} M^{\tens n}
$$
is left adjoint to the forgetful functor from the category of monoids in $\M$ to $\M$.
\end{proposition}

\begin{proof}
This is folklore. The proof relies on the assumption that $\tens$ distributes over countable coproducts to show that $\mor_\Alg(\F(M), X)\cong \mor(M,UX)$. 
\end{proof}

\begin{remark}
Instead of $\V$-categories, we could have called $\V$-quivers the objects of $\V^{S\times S}$. In the case where $\V$ is the category of vector spaces over a field $\k$, We can work with $\oplus Q= \oplus_{i,j\in S} Q(i,j)$ instead of a precategory/quiver $Q$, and $\oplus \F(Q)$ is then the classical path algebra associated to $Q$. Note also that the category of representations of $Q$ is the functor category $_\k\Mod^{\F(Q)}$.
\end{remark}

Now is a good time to talk about colimits. We already mentioned that $\Precat = \V^{S\times S}$, so $\Precat$ has all (small) limits and colimits, defined pointwise from $\V$. Since $\Precat$ is monoidal biclosed, $\circ$ distributes over colimits which means $U:_R\Mod \rightarrow \Precat$ creates all (small) colimits, something that is true even in the general context of section 1. In this case, we get that all colimits exist in $_R\Mod$, $_R\Mod_R$, and are pointwise. The other requirement of Proposition \ref{prop:freemonoid} is that $\circ_R$ distributes over countable coproducts, which holds because $_R\Mod_R$ is monoidal biclosed by Proposition \ref{prop:BimodulesMonoidalBiclosed}. Therefore we can apply the proposition and get a free $R$-algebra functor $F_R$. And of course, given $C$ an $R$-algebra, we also have free module adjunctions between $_R\Mod_R$, $_C\Mod$ and $\Mod_C$.

\subsection{Model structures}

We now apply the results from \cite{hovey2} and \cite{batanin-berger} mentioned in section 1.2 to get model structures on $\Cat$, $\Mod_R$ and the various categories of modules considered.

Given $i,j\in S$, consider the functor $(-)_{i,j}: \V \rightarrow \Precat$ sending an object of $\V$ to the precategory consisting in $\V$ concentrated in level $(i,j)$, explicitly $V_{i,j}(k,l)$ is $V$ for $(k,l)=(i,j)$ and $*$ otherwise. Note that this functor has a right adjoint in the projection sending a precategory $X$ to $X(i,j)$. Since $\Precat$ is defined as $\V ^{S\times S}$, it inherits a pointwise model structure from $\V$, which by Proposition 1.11.10 in \cite{hirschhorn} is cofibrantly generated with $I_\Precat$ defined as the union of image sets $\cup_{(i,j)\in S\times S}(I_\V)_{i,j}$, and similarly for $J_\Precat$.

\begin{proposition}
If $\V$ satisfies the monoid axiom, so does $\Precat$, and if $\V$ is compactly generated, so is $\Precat$.
\end{proposition}

\begin{proof}
It is relatively obvious that the category of precategories satisfies the monoid axiom, given that transfinite compositions and pushouts are defined levelwise in $\Precat$, which means it is enough to observe that for $X$ any precategory, $J_\Precat \circ X$ is composed of maps whose levels are either $*\rightarrow *$ or $\psi \Smash X(k,j)$ for some $\psi$ in $J_\V$. These are weak equivalences in $\V$ by assumption.

Assume $\V$ is compactly generated (see Definition \ref{def:CompactlyGenerated}). We already know that $\Precat$ is cofibrantly generated, and since colimits and morphisms are defined levelwise, the condition of smallness relative $I_\Precat$ boils down to smallness relative $I_\V$. Finally, weak equivalences and colimits are determined pointwise and $I_\Precat^\circ$ is actually just $\Pi_{S\times S}I_\V^\Smash$, so the class of weak equivalences of $\Precat$ is $I_\Precat^\circ$-perfect.
\end{proof}

This result coupled with Propositions \ref{prop:HoveyTransfer} and \ref{prop:BataninBergerTransfer} yield transfered projective model category structures on $_R\Mod$, $\Mod_R$, $_R\Mod_R$ and $\Cat$. Most of the symmetric monoidal model categories one would want to consider for the choice of $\V$ satisfy the monoid axiom. 

\begin{remark}
We will mostly work with the assumption that the ``ground category" $R$ is cofibrant (this terminology being motivated by the analogy with the ground field in classical Koszul duality of algebras). This is reasonable, in part because of the applications we have in mind, in part because it yields desirable properties for the categories of $R$-modules and bimodules. For example, Theorem 2.4 in \cite{hovey2} which says that under mild conditions, a weak equivalence of cofibrant monoids induces a Quillen equivalence between the corresponding categories of modules over the source and the target, which is not necessarily true if the monoids are not cofibrant (see Example 2.5 of the same source).
\end{remark}

\begin{corollary}
If $\V$ satisfies the monoid axiom and $R$ is cofibrant, so does $_R\Mod_R$.
\end{corollary}

\begin{proof}
This is a consequence of something that holds in the general context of section 1: if $\C$ is a monoidal model category satisfying the monoid axiom, and $R$ is a cofibrant monoid in it, then $_R\Mod_R$ also satisfies it.

Let's simply write $J_R$ for the generating trivial cofibration set of $_R\Mod_R$, and let $X$ be an $R$-bimodule. A generating trivial cofibration is of the form $R\circ f \circ R$ where $f:A\rightarrow B$ is a generating trivial cofibration of $\Precat$. An element of the set $J_R\circ_R X$ is therefore of the form $R\circ f \circ R \circ_R X = R\circ f \circ X$, and because $R$ is cofibrant the functor $R\circ -$ is left Quillen, which means that $R\circ f$ is a trivial cofibration (using the pushout version of the monoidal model category axiom with $*\rightarrow R$ and $A\rightarrow B$). Which means that $_R\Mod_R$ satisfies the monoid axiom when $\Precat$ does.
\end{proof}

We can therefore use the general definition of the Koszul dual given in section 1.3 in this context of precategories and modules: for $C$ an augmented $R$-algebra,
$$
K(C)=\hom^r_C(B(R,C,C),B(R,C,C))
$$
Note that the simplicial bar construction $B_\bullet(R,C,C)$ consists of chains of composable arrows from $C$.

\begin{remark}
It is useful to clarify the situation of the geometric realization. In section 1.3 (Remark \ref{remark:GeomReal} and above), we defined the geometric realization of a Reedy cofibrant simplicial $R$-bimodule $X_\bullet$ as the coend of a cofibrant approximation of $X_\bullet$ when seen as a constant cosimplicial simplicial object $cc_*X$. However, in the cases where $\V$ is a suitable category of spaces or spectra like $\CG$ or the category of $\bbS$-modules, there is an external tensor product enabling us to write $$
\vert X \vert = \int^{n\in \Delta} \Delta^n \smash X_n
$$
(as mentioned in Remark \ref{remark:GeomReal}). To be precise, this holds because $\Delta^\bullet \smash X_\bullet$ is a cofibrant approximation of $cc_*X$ in $\mathrm{cs}{_R\Mod_R}$.
\end{remark}

\subsection{A few examples}

The main purpose of our approach is to get an extended definition of Koszul duality, thus we are taking a special interest in the case where $R$ is the category $\Sigma$ where $\Sigma(i,j)$ is the set of bijections from $\{1,\ldots, i\}$ to $\{1,\ldots, j\}$ (with the appropriate modifications to make it an element of $\V$). In section 3, we will justify that our definition of Koszul dual extends the one for Koszul duality of operads in spectra due to Michael Ching, by taking $R=\Sigma$ and $C$ the prop associated to an operad $P$. But there are of course other notable examples.

Let us say the indexing set $S$ has one element. In that case, precategories are just objects of $\V$ and the definition of Koszul dual we study in this article corresponds to a notion of duality in $\V$, for example as in \cite{DGI}. Another natural example to consider is when both $R$ and $C$ only have endomorphisms, which implies that $B(R,C,C)(i,j)$ is trivial for $i\neq j$ so $B(R,C,C)$ only has endomorphisms as well, and the same is true for the Koszul dual, with
$K(C)(i,i) = \map_{C(i,i)}(B(R,C,C)(i,i), B(R,C,C)(i,i))=K(C(i,i))$.


Let us now consider the case $R=\Sigma$. Let $\V$ be a suitable category of spaces, like the category $\CG_+$ of compactly generated weak Hausdorff pointed spaces. To simplify the notation, whenever $m$ is a natural number, we will simply denote by $m$ the basic set with $m$ elements $\{0,1,\ldots, m-1\}$ (instead of using the traditional bold \textbf{m} or underlined $\underline{m}$). As mentioned above, the case where $C$ is the prop of an operad is discussed in section 3, so let us consider the case $C=FI$, the skeleton of the category of finite sets and injections. That is, $FI(m,n)= \inj (m,n)_+$ if $m\le n$ and $FI(m,n) = *$ otherwise. Note that $FI$ is a $\Sigma$-algebra. Although we cannot compute the Koszul dual of $FI$, there is a nice simple description of its bar construction levelwise.

\begin{proposition}
The bar construction $B(\Sigma, FI, \Sigma)$, which we will denote $B(FI)$ for simplicity, is isomorphic in level $(m,n)$ to
$$
\inj(m,n)_+ \smash S^{n-m}
$$
\end{proposition}

\begin{proof}
By definition, the $l$-th simplicial degree of the bar construction is
$$
B_l(FI) (m,n)= (\Sigma \circ_\Sigma FI^{\circ_\Sigma^l}\circ_\Sigma \Sigma)(m,n) = FI^{\circ_\Sigma ^l}(m,n)
$$
so its non degenerate part is
$$
\left(\coprod_{m<n_1<\cdots<n_{l-1}<n} \inj(m,n_1) \times_{\Sigma_{n_1}} \cdots \times_{\Sigma_{n_{l-1}}} \inj(n_{l-1},n)\right)_+
$$
A nontrivial element of this space is an equivalence class of composable chains of injections
$$
m \xrightarrow{i_1} n_1 \xrightarrow{i_2} \cdots \xrightarrow{i_{l-1}} n_{l-1} \xrightarrow{i_l} n
$$
under the actions of the relevant symmetric groups. Such an equivalence class can be seen as the data of an injection $i:m\rightarrow n$ together with a chain
$$
i(m)\subset X_1 \subset X_2 \subset \cdots \subset X_{l-1} \subset n
$$
of subsets of $n$.
A homeomorphism between these two spaces can be defined by setting $i=i_l \cdots i_1$ and $X_j$ as the image of $i_l \cdots i_{j+1}$. Looking at how the face and degeneracy maps of the bar construction translate to the latter description, we can deduce that
$$
B(FI)(m,n) = \inj(m,n)_+ \smash \vert \P(n-m) \vert /(\vert \P_\subsetneq (n-m)\vert \cup \vert \P_\supsetneq (n-m)\vert)
$$
where $\P(n-m)$ denotes the poset of subsets of $n-m$, and $\P_\subsetneq (n-m)$ (resp. $P_\supsetneq (n-m)$) denotes its subposet consisting of all chains whose first element is not the empty set (resp. whose last element is not the full set $n-m$). Using the model of $\vert \P(n-m) \vert $ as a cube of dimension $n-m$ we see that the right-hand quotient in the above formula corresponds to the quotient of the $(n-m)$-dimensional cube by its boundary, that is, to $S^{n-m}$.
\end{proof}

\section{Connection to Koszul duality of operads}

We now explore the case of $R=\Sigma$ further, with the idea of relating our definition of the Koszul dual of a category to Koszul duality of operads. In this context, props are examples of $\Sigma$-algebras, and in particular props of operads.

\begin{remark}
This is a good time for a small interlude on what props are and how they fit in the setting of (pre)categories with a fixed object set $S=\langle \mathscr C \rangle$, meaning the set of words on a given set $\mathscr C$ of colors. We can define $\mathscr C$-colored props to be strict symmetric monoidal model categories with object set $S$, and this coincides with, for example, definition 2.2.13 in \cite{Yau08}. In the latter, props are symmetric bimodules equipped with compatible horizontal and vertical composition laws, one corresponding to the composition of morphisms i.e. the category structure, and the other to the monoidal structure. As before, we will focus on the case where $S=\N$, and thus on 1-colored props, but we believe our results could generalize to props with multiple colors.
\end{remark}

\subsection{The Prop functor is monoidal}

In what follows we will assume that the base category $\V$ is a suitable category of spectra, for example the category of EKMM spectra.


\begin{definition}
To every symmetric sequence $X$ we can associate a $\Sigma$-bimodule
$$
\prop X(m,n)= \coprod_{f: m \rightarrow n} X(m_1)\Smash \cdots \Smash X(m_n)
$$
where $m_i$ denotes the number of elements in $f^{-1}(i)$ (see Example 60 in \cite{Markl_operads_props}). If $X$ is an operad, $\prop X$ is a prop (meaning that, admittedly, our notation can be a bit misleading if $X$ is not).
\end{definition}
If $X$ is reduced, meaning $X(0)=*$ and $X(1)= S^0$, then $\prop X(m,n)=*$ for $m<n$ (we say $\prop X$ is {\it directed}) and $\prop X(n,n)=\Sigma_n$.
\begin{remark}
 In the case where $\V$ is, for example, the category of dg-modules, it can be interesting to note that directed precategories are exactly representations in $\V$ of the infinite quiver $0\leftarrow 1 \leftarrow \cdots$.
\end{remark}

Just like in the context of operads, these properties are very useful for a precategory to have, one example among others is that they make the description of composition products of $\prop X$ much nicer. Namely the indexing set of the coproduct in $(\prop X \circ_\Sigma \prop X)(m,n)$ becomes finite. We will therefore assume that $X$ is reduced throughout this section, and in the more general context of section 5 we will also work with a directed $R$-category $C$.

We investigate how the Koszul dual $\Sigma$-algebra of the prop of an operad $P$ relates to the prop of its Koszul dual. The main thing to look out for being that, whereas the Koszul dual construction from section 1 makes use of an endomorphism object and the multiplication is the composition of endomorphisms, we do not have access to hom objects in the context of operads, so the Koszul dual operad of $P$ is defined as the dual of $B(1,P,1)$.

\begin{proposition}\label{prop:PropFunctorMonoidal}
The functor $\prop: \V^{\Sigma} \rightarrow {_\Sigma\Mod_\Sigma}$ is monoidal.
\end{proposition}

\begin{proof}
On one hand, we have
\begin{align*}
    &(\prop(M)\circ_\Sigma \prop(N))(i,j)\\ &= \bigvee_k \prop(M)(i,k) \Smash_{\Sigma_k} \prop(N)(k,j) \\
    &=\bigvee_k \big(\bigvee_{i\rightarrow k} M(i_1) \Smash \cdots \Smash M(i_k) \big)\Smash_{\Sigma_k}  \big( \bigvee_{k\rightarrow j} N(k_1) \Smash \cdots \Smash N(k_j) \big)
\end{align*}
Let us consider the groupoid $\C_{i,j}$ whose objects are diagrams $i\rightarrow k \rightarrow j$ and where $\C_{i,j}((g,f),(g',f'))$ consists of the bijections $k\rightarrow k$ making the natural diagram
$$
\xymatrix{
i \ar[r]^-{g'} \ar[d]_{g} &k \ar[d]^{f'} \\
k \ar[r]^-{f} \ar[ru]^-{\sigma} &j
}
$$
commute. Notice that all morphisms can be written as $\sigma: (g,f) \rightarrow (\sigma g, f \sigma^{-1})$. Then the formula we obtained above can be seen as the colimit of the functor $F:\C_{i,j} \rightarrow \V$ sending $i\rightarrow k \rightarrow j$ to
$$
M(i_1) \Smash \cdots \Smash M(i_k) \Smash N(k_1) \Smash \cdots \Smash N(k_j)
$$
On the other hand,
$$
\prop(M\circ_\Sigma N)(i,j) = \bigvee_{i\rightarrow j} (M\circ_\Sigma N)(i_1)\Smash \cdots \Smash (M\circ_\Sigma N)(i_j)
$$
with
$$
(M\circ_\Sigma N)(i_l) = \bigvee_{k_l} \bigvee_{i_l\rightarrow k_l} M(i_{l,1}) \Smash \cdots \Smash M(i_{l,k_l}) \Smash_{\Sigma_{k_l}} N(k_l)
$$
Consider now the full subcategory $\D_{i,j}$ consisting of the objects $(g,f): i\rightarrow k \rightarrow j$ of $\C_{i,j}$ where the right-hand side morphism $f$ is order-preserving. Notice that for $(g,f)$ an object of $\D_{i,j}$, a morphism $\sigma: (g,f) \rightarrow (\sigma g, f \sigma^{-1})$ lies in $\D_{i,j}$ if and only if $\sigma$ preserves $f^{-1}(x)$, for all $1\le x\le j$. From this we can deduce that taking the colimit of $F\vert_{\D_{i,j}}$ gives the formula we have above for $\prop (M\circ_\Sigma N)(i,j)$. Since $\D_{i,j}$ and $\C_{i,j}$ are equivalent categories, we get the desired isomorphism $\prop (M\circ_\Sigma N) \cong \prop(M)\circ_\Sigma \prop(N)$.
\end{proof}

\begin{remark}\label{remark:OperadsNotClosed2}
Coming back to Remark \ref{remark:OperadsNotClosed1}, one can identify a symmetric sequence $M$ to $\prop(M)(-,1)$, and using Proposition \ref{prop:PropFunctorMonoidal}, a map $M\circ_\Sigma N\rightarrow P$ to a collection of maps
$$
\prop(M)(i,k)\smash_{\Sigma_k} \prop(N)(k,1) \rightarrow \prop (P)(i,1)
$$
which is the same as a collection of maps $\prop(M)(i,k) \rightarrow \map(N(k), P(i))$. The obstacle to the existence of homomorphism objects for symmetric sequences is then that this does not translate to a map where the source is $M$. It is worth noting, in that regard, that while the functor $\prop$ is monoidal it has no left or right adjoint.
\end{remark}

\subsection{An equivalence of $\Sigma$-bimodules}

We will now use the fact that the functor $\prop$ is monoidal to exhibit a zigzag of weak equivalences of $\Sigma$-bimodules between $K(\prop P)^\op$ and $\prop K(P)$. This is weaker than the result we obtain in section 3.3 (an actual zigzag of weak equivalences of {\it $\Sigma$-algebras}), but it is worth presenting since its construction is much more natural.

\begin{lemma} \label{Corollary:IntermediateStepCompatibility}
There is an equivalence of $\Sigma$-bimodules
$$
K(\prop P) \simeq \hom_\Sigma(\prop(B(1,P,1)), \Sigma)
$$
\end{lemma}

\begin{proof}
Note that $\prop 1 = \Sigma$. By Proposition \ref{prop:DescriptionDual} it is enough to show that
$$
B(\prop 1, \prop P, \prop 1) = \prop B(1,P,1)
$$
which follows from the Proposition \ref{prop:PropFunctorMonoidal} and the fact that the symmetric monoidal product $\Smash$ of the base category $\V$ distributes over colimits.
\end{proof}

\begin{proposition}\label{prop:compatibility}
Assume that $P(n)$ is finite for all $n$, by which we mean it is a finite complex (i.e. dualizable). There is a zigzag of equivalences of $\Sigma$-bimodules
$$
K(\prop P)^{op} \simeq \prop (K(P))
$$
\end{proposition}

\begin{proof}
Because of Corollary \ref{Corollary:IntermediateStepCompatibility}, it is sufficient to show that for a symmetric sequence $M$ such that $M(n)$ is finite for all $n$, $\hom_\Sigma(\prop (M), \Sigma)(i,j)$ is equivalent to $\prop(M^\vee)(j,i)$. Here by $M^\vee$ we mean the symmetric sequence defined by $M^\vee(n)=M(n)^\vee$, which is how we write the Spanier-Whitehead dual $F(M(n),\bbS)$ of $M(n)$ (see part III in \cite{ESHT} for a treatment of equivariant stable duality). First, note that $\hom_\Sigma(\prop (M), \Sigma)(i,j)$ is equivalent to
\begin{align*}
    \prod_k \map_{\Sigma_k} (\prop(M)(j,k), \Sigma(i,k)) &\simeq \map_{\Sigma_i}(\prop(M)(j,i), \Sigma_i) \\
    &\simeq \prop(M)(j,i)^\vee
\end{align*}
because the symmetric group $\Sigma_i$ is self-dual (let us emphasize that this is in no way an isomorphism). From here, we can conclude using the finiteness assumption:
\begin{align*}
    \prop(M)(j,i)^\vee &\simeq \left( \bigvee_{f:j\rightarrow i}M(j_1) \Smash \cdots \Smash M(j_i) \right)^\vee \\
    &\simeq \bigvee_{f:j\rightarrow i}M(j_1)^\vee \Smash \cdots \Smash M(j_i)^\vee
    &\simeq \prop(M^\vee)(j,i)
\end{align*}
\end{proof}



\subsection{An equivalence of $\Sigma$-algebras}

In this section, we want to get a zigzag of weak equivalences of {\it algebras}
$$
\vp: \prop (BP^\sharp)^{\op} \sim K(\prop P)
$$
for a reduced, levelwise finite operad $P$ in spectra (meaning, $\bbS$-modules), thus strengthening the compatibility result of Proposition \ref{prop:compatibility}.

By definition of $K(\prop P)$ as the derived $\prop P$-endomorphism object of $\Sigma$, a natural way to tackle this problem is to define an appropriate $(\prop (BP^\sharp)^{\op}, \prop P)$-bimodule $M$ that is equivalent to $\Sigma$ (as a right $\prop P$-module). We will first show that $M=\prop (B(P,P,1)^\sharp)^{\op}$ satisfies this.

\begin{definition}
Let $X$ and $Y$ be symmetric sequences. Define $X\bar \circ Y$ to be the symmetric sequence with
$$
(X \bcirc Y)(i)= \bigvee_{k\rightarrow i} X(k_1) \smash \cdots \smash X(k_i) \smash_{\Sigma_k} Y(k)
$$
Note that this is simply the definition of the standard composition product $X\circ Y$, with the indexing set altered to be maps from $k$ to $i$ instead of from $i$ to $k$.
In addition, we define $X\hcirc Y$ by
$$
(X\hcirc Y)(i)= \prod_{i\rightarrow k} X(i_1) \smash \cdots \smash X(i_k) \smash_{\Sigma_k} Y(k)
$$
This definition can be found in remark 2.20 of \cite{Ching05}.
\end{definition}

\begin{remark}\label{remark:CompProdCompatibilitySec3}
It is easy to see that the composition product $\bcirc$ behaves well with respect to the functor $\prop: \V^\Sigma \rightarrow \V^{\Sigma \times \Sigma}={_\Sigma\Mod_\Sigma}$, by which we mean that
$$
\prop(X\bcirc Y) \cong (\prop X)^{\op} \circ (\prop Y)
$$
naturally in $X$ and $Y$. In addition, even though $\bcirc$ is not associative, there is a natural isomorphism
$$
(X\circ Y) \bcirc Z \cong Y\bcirc (X \bcirc Z)
$$
Because of this, if $P$ is an operad i.e. a monoid for $\circ$ in the category of symmetric sequences, there is a natural notion of {\it left $P$-$\bcirc$-module}.
\end{remark}

\begin{definition}
A left $P$-$\bcirc$-module structure on a symmetric sequence $X$ is a map $P\bcirc X \rightarrow X$ such that
$$
\xymatrix{
(P\circ P) \bcirc X \ar[r]^-{\cong} \ar[d] & P\bcirc (P\bcirc X) \ar[r] &P\bcirc X \ar[d] \\
P\bcirc X \ar[rr]&& X
}
$$
Let us emphasize that the isomorphism used on the top row swaps the two copies of $P$.
\end{definition}

\begin{proposition}
Let $X$ be a symmetric sequence and suppose that $X$ is a left $P$-module for some operad $P$. Then its levelwise dual $X^\sharp$ is a left $P$-$\bcirc$-module. On the other hand, if $Q$ is a cooperad and $X$ is a right $Q$-comodule, there is a natural right $Q^\sharp$-module structure on $X^\sharp$.
\end{proposition}

\begin{proposition}
Let $X$ be a symmetric sequence. If $X$ is a left $P$-$\bcirc$-module, then $\prop X$ is a left $(\prop P)^{\op}$-module.
\end{proposition}

Let us now come back to the problem we are interested in, fixing a reduced operad $P$ with finite levels. We know from section 7 of \cite{Ching05} that the bar construction $B(P,P,1)$ is a left $P$-module and a right $BP$-comodule. Therefore, from the results above, $\prop (B(P,P,1)^\sharp)$ is a $((\prop P)^{\op}, \prop (BP^\sharp))$-bimodule (the fact that the left and right module structure commute is straightforward). Equivalently, $\prop (B(P,P,1)^\sharp)^{\op}$ is a $({\prop (BP^\sharp})^\op, \prop P)$-bimodule. This can be stated equivalently as follows.
\begin{proposition}
There is a map of $\Sigma$-algebras
$$
\prop (BP^\sharp)^{\op} \rightarrow \mathrm{end}_{\prop P}^r (\prop (B(P,P,1)^\sharp)^{\op})
$$
\end{proposition}

Let us fix $\Q: \Alg_\Sigma\rightarrow \Alg_\Sigma$ a chosen cofibrant replacement functor for the category of $\Sigma$-algebras, and fix a cofibrant replacement functor $(-)_\cof$ for the category of $(\Q\prop(BP^\sharp)^\op,\prop P)$-bimodules. Then $(-)_\cof$ is also a cofibrant replacement functor in the category of right $\prop P$-modules. Indeed, on one hand, the forgetful functor is right Quillen. On the other hand, due to the fact that $\Q\prop(BP^\sharp)^\op$ is in particular cofibrant as a $\Sigma$-bimodule, every generating cofibration of the category of $(\Q\prop(BP^\sharp)^\op,\prop P)$-bimodules is a cofibration of right $\prop P$-modules as well. We thus get an algebra map
$$
\vp : \Q\prop (BP^\sharp)^{\op} \rightarrow \R\mathrm{end}_{\prop P}^r (\prop (B(P,P,1)^\sharp)^{\op})
$$


\begin{remark}
There is no need to consider $\Q \prop (BP^\sharp)^{\op}$ instead of $\prop (BP^\sharp)^{\op}$ if there exists a functorial cofibrant replacement on $\Mod_{\prop P}$ that is $_\Sigma\Mod_\Sigma$-enriched. In general this is quite a strong condition to ask (see \cite{RiehlCat}, section 13.2 for conditions for the existence of enriched functorial factorizations), so we will not consider this case. Let us also mention that instead of the cofibrant replacement functor $(-)_\cof$, we could use $B(-,\prop P,\prop P): \Mod_{\prop P} \rightarrow \Mod_{\prop P}$. Although it is merely a functorial cofibrant {\it approximation} functor (meaning the map from it to the identity functor is an object-wise weak equivalence but not necessarily an object-wise fibration), it satisfies the properties we need, mostly because it has the nice property that there is a natural isomorphism $B(X\circs Y, \prop P, \prop P)\cong X\circs B(Y,\prop P,\prop P)$ so any left module structure on a right $\prop P$-module $Y$ induces a left module structure on $B(Y,\prop P,\prop P)$.
\end{remark}

\begin{remark}\label{remark:PropsOfOperadsAreCofibrant}
Notice that if $M$ is a symmetric sequence, then the right $\Sigma_j$-action on $(\prop M)(i,j)$ is free, for any $i,j$ (because it comes from the action of $\Sigma_j$ on the set of surjections from $i$ to $j$ which is free). In particular, if $M$ is cofibrant as a symmetric sequence, $\prop M$ is a cofibrant right $\Sigma$-module.
\end{remark}

\begin{remark}
 Whenever $A$ is a cofibrant left  $\Sigma$-module and $C$ a $\Sigma$-algebra, the functor $\hom^l_\Sigma(A,-): \Mod_C \rightarrow \Mod_C$ is right Quillen, therefore $A\tens_\Sigma -: \Mod_C \rightarrow \Mod_C$ is left Quillen, and dually if $A$ is a left $\Sigma$-module.
\end{remark}

\begin{lemma}\label{lemma:PropDual}
Let $M$ be a levelwise finite symmetric sequence. Then there is a natural weak equivalence
$$
\prop(M^\sharp)^\op \xrightarrow{\sim} \hom_\Sigma (\prop M, \Sigma)=_\mathrm{def} D(\prop M) 
$$
of $\Sigma$-bimodules.
\end{lemma}

\begin{proof}
The map is induced by
$$
\bigvee_n M(n)^\sharp \smash_\Sigma M(n) \rightarrow \S
$$
which can be seen as a map of symmetric sequences $M^\sharp \bcirc M \rightarrow 1$, and applying the functor $\prop: \V^{\N} \rightarrow \V^{\N\times \N}$, which is monoidal by Proposition \ref{prop:PropFunctorMonoidal}. It  is an equivalence because it fits in the commutative diagram
$$
\xymatrix{
\prop (M^\sharp)^\op \ar[r] \ar[rd]_-\sim &\hom_\Sigma(\prop M, \Sigma) \ar[d]^\sim\\
&(\prop M)^\sharp
}
$$
\end{proof}

The canonical map $B(P,P,1)\rightarrow 1$ obtained from the augmentation of $P$ induces a map
$$
\Sigma \rightarrow (\prop B(P,P,1)^\sharp)^\op
$$
and it is an equivalence of right $\prop P$-modules. In addition, Lemma \ref{lemma:PropDual} gives us a weak equivalence of right $\prop P$-modules
$$
\prop(B(P,P,1)^\sharp)^\op)_\cof \xrightarrow{\sim} D(\prop B(P,P,1))
$$
We will show that the following diagram is commutative, and, therefore, that $\vp$ is a weak equivalence. The diagonal arrow is the map from Lemma \ref{lemma:PropDual} with $M=\prop BP$.
$$
\xymatrix{
\Q\prop(BP^\sharp)^\op \ar[r]^-{\vp} \ar[d]^\sim &\mathrm{end}_{\prop P}^r(\prop (B(P,P,1)^\sharp)^\op_\cof) \ar[d]^\sim\\
\prop(BP^\sharp)^\op \ar@/_1pc/[rdd]_-\sim &\hom^r_{\prop P} (\Sigma_\cof, D(\prop B(P,P,1))) \ar[d]^\cong \\
&\hom^r_{\Sigma} (\Sigma_\cof \circ_{\prop P} \prop B(P,P,1), \Sigma)\\
&\hom_\Sigma^r(\Sigma \circ_{\prop P} \prop B(P,P,1), \Sigma)= D(\prop BP) \ar[u]_\sim \\
}
$$

This is equivalent to proving that the adjoint diagram below commutes.
$$
\xymatrix@C=1em{
&\Q\prop(BP^\sharp)^\op\circs\Sigma_\cof \circ_{\prop P} \prop B(P,P,1) \ar[d] \ar@/_2em/[ldd] \\
&\Q\prop(BP^\sharp)^\op\circs\prop(B(P,P,1)^\sharp)_\cof \circ_{\prop P} \prop B(P,P,1) \ar[d]\\
\prop (BP^\sharp) \circs \prop BP \ar[rd] &\prop (B(P,P,1)^\sharp)^\op_\cof \circ_{\prop P} \prop B(P,P,1) \ar[d] \\
&\Sigma
}
$$


We will denote $\alpha$ and $\beta$ the left-hand and right-hand composites of this diagram respectively. Remark that if we take $(-)_\cof$ to be a functorial cofibrant replacement in the category of $(\Q\prop(BP^\sharp)^\op,\prop P)$-bimodules obtained from the functorial factorization functors of the model category structure, it comes equipped with a natural transformation $(-)_\cof \rightarrow \Id$. Hence there is a commuting square
$$
\xymatrix{
\Sigma_\cof \ar[d]^\sim \ar[r] &\prop(B(P,P,1)^\sharp)^\op_\cof \ar[d]^\sim\\
\Sigma \ar[r]&\prop(B(P,P,1)^\sharp)^\op
}
$$
Using this together with the fact that
$$
\prop(B(P,P,1)^\sharp)^\op_\cof \rightarrow \prop(B(P,P,1)^\sharp)^\op
$$
is in particular a left $\Q\prop(BP^\sharp)^\op$-module map, we get a commuting diagram of right $\prop P$-modules
$$
\xymatrix{
\Q\prop(BP^\sharp)^\op \circs \Sigma_\cof \ar[r] \ar[d] &\Q\prop(BP^\sharp)^\op \circs \Sigma \ar[d]\\
\Q\prop(BP^\sharp)^\op \circs \prop(B(P,P,1)^\sharp)^\op_\cof \ar[r] \ar[d]&\Q\prop(BP^\sharp)^\op \circs \prop(B(P,P,1)^\sharp)^\op \ar[d] \\
\prop (B(P,P,1)^\sharp)^\op_\cof \ar[r]& \prop( B(P,P,1)^\sharp)^\op
}
$$
and the problem is reduced to showing that
$$
\xymatrix{
\prop(BP^\sharp)^\op \circs \Sigma \circ_{\prop P} \prop B(P,P,1) \ar[r] \ar[dr]&\Sigma \\
&\prop(B(P,P,1)^\sharp)^\op \circ_{\prop P} \prop (B(P,P,1)) \ar[u]
}
$$
commutes, where the diagonal map consists in applying $\Sigma \rightarrow \prop(B(P,P,1)^\sharp)^\op$ and using the left $\prop(BP^\sharp)^\op$-module structure on $\prop(B(P,P,1)^\sharp)^\op$. The other two maps are obtained as in Lemma \ref{lemma:PropDual}.

Using the compatibility of $\prop$ with the composition products of symmetric sequences $\circ$ and $\bcirc$ (see \ref{prop:PropFunctorMonoidal} and \ref{remark:CompProdCompatibilitySec3}), it is enough to show that the following square diagram of symmetric sequences commutes.
$$
\xymatrix{
BP^\sharp \bcirc_P B(P,P,1) \ar[r]^-\cong \ar[d] &BP^\sharp \bcirc BP \ar[d]\\
B(P,P,1)^\sharp \bcirc_P B(P,P,1) \ar[r] &1
}
$$

\begin{remark}
The notation $M \bcirc_P N$ makes sense when $M$ is a {\it left} $P-\bcirc$-module and $N$ is a left $P$-module, and the evaluation map associated to the dual of a symmetric sequence with a left $P$-module structure factors through this $P$-equivariant composition product. In addition, if $X$ is a spectrum seen as a symmetric sequence concentrated in arity 1, and we define $\map(N,X)(i)= \map(N(i),X)$, then there is an exponential correspondence between maps of symmetric sequences $M\bcirc_P N\rightarrow X$ and maps of left $P-\bcirc$-modules $M\rightarrow \map(N,X)$.
\end{remark}

Using this remark, the square diagram above is equivalent to
$$
\xymatrix{
BP^\sharp \ar[d]_{g^\sharp} \ar[rd]^{f^\sharp} \\
B(P,P,1)^\sharp \ar[r]^\id &B(P,P,1)^\sharp
}
$$
where $f: B(P,P,1)\rightarrow BP$ uses the augmentation of $P$ to collapse the leftmost $P$ factor in each simplicial level, and
$$
g:B(P,P,1) \rightarrow B(P,P,1) \hcirc BP \rightarrow 1 \hcirc BP = BP.
$$
This diagram commutes by definition of the $BP$-comodule structure on $B(P,P,1)$ defined in section 7.3 of \cite{Ching05}. Therefore, we can improve Proposition \ref{prop:compatibility} to get the following result.

\begin{proposition}\label{prop:StrongCompatibility}
Let $P$ be a reduced, levelwise finite operad in (EKMM) spectra. Then the map
$$
\vp: \Q\prop (BP^\sharp)^\op \rightarrow \R\mathrm{end}^r_{\prop P}(\prop (B(P,P,1)^\sharp)^\op)
$$
is a weak equivalence of $\Sigma$-algebras.
\end{proposition}

\section{External modules}


In this section, we recall a standard  way of seeing functors as (external) modules over a category and discuss the link between the categories of external modules over $C$ and $K(C)$. Just like we discussed in section 2 a notion of precategory as an object that when endowed with an algebra structure becomes a category, we now consider {\it prefunctors}, objects that when equipped with a module structure over a category $C$ become functors $C\rightarrow \V$. This approach of seeing functors as modules is classical, and is discussed in section 6 of \cite{SchSh03} for instance. There are several reasons for looking at these external modules instead of the "internal" ones considered previously. First, this is the natural setting for defining enriched functors. In addition, a structure of external module over the prop of an operad is the same as a left module structure over that operad, as stated in Proposition \ref{prop:ExternalModulesPropOperad}.

We will discuss Quillen equivalences between categories of $C$- and $K(C)$-modules, namely Corollary \ref{cor:QuillenEq1} and Proposition \ref{prop:DoubleDualExternalModule}. The former is implied by a result of \cite{SchSh03}, while the latter is new.

\subsection{Definitions}

\begin{definition}
As before, S is a fixed set which we call the object set. We call $\V^S$ the category of prefunctors and define external tensor products over the category of precategories $\Precat$ by
$$
(C\tens Y)(i)= \bigvee_{j\in S} C(i,j)\Smash Y(j)
$$
and
$$
(Y\tens C)(i)= \bigvee_{i\in S} Y(j) \Smash C(j,i)
$$
for $C\in \Precat$ and $Y\in \V^S$.
\end{definition}

Note that $C^{op}\tens Y\cong Y\tens C$ unlike for the internal composition product in $\Precat$, which means that $Y$ is an external right $C$-module if and only if it is an external left $C^{op}$-module. The tensor product satisfies natural compatibility properties with respect to the monoidal product $\circ$ of $\Precat$, so that $\V^S$ becomes what is called a category {\it tensored over} $\Precat$. The term {\it prefunctor} is justified by the fact that for $Y\in \Prefun$, a right $C$-module structure on $Y$ is the same as a structure of $\V$-enriched functor $Y: C\rightarrow \V$, something that is easily checked using the hom-tensor adjunction property in $\V$. Moreover, right $C$-module morphisms correspond to $\V$-natural transformations. These categories of external modules are equipped with model structures transferred from the levelwise model structure on $\Prefun$, in a similar way as in Section 1.

\begin{remark}
Given $R$ a category with object set $S$, an internal right $R$-module, as defined in section 2, is not the same thing as an external right $R$-module. However, internal $R$-bimodules and external $(R^{op}\times R)$-modules are the same.
\end{remark}

\begin{proposition}\label{prop:ExternalModulesPropOperad}
Given an operad $P$ in $\V$, an external $\prop P$-module structure is the same as a left $P$-module structure.
\end{proposition}

We also have copowers which we may call {homomorphism prefunctors}, defined by
\begin{align*}
\homl(C,Y)(i)&= \prod_{j\in S} \map(C(j,i), Y(j)) \\
\homr(C,Y)(i)&= \prod_{j\in S} \map(C(i,j), Y(j))
\end{align*}
for $C$ a precategory, $Y$ a prefunctor, and $i$ in $S$; note that $\homl(C,Y)=\homr(C^{op}, Y)$. In addition, associated to a pair of prefunctors $Y,Z$ is a {\it homomorphism precategory} with
$$
Z^Y(i,j)= \map(Y(i), Z(j))
$$
which is a category in the $Y=Z$ case. One can also note that $\Prefun(Y,Z)= \Precat(\mathbb{1}, Z^Y) = \prod_{i\in S} Z^Y(i,i)$, which is the same relationship as the one between morphisms and homomorphisms of precategories. This homomorphism precategory gives $\Prefun$ a $\Precat$-enriched category structure. The tensor products and hom-objects defined above are linked by the following natural isomorphisms, which justifies the use of the term {\it copowers} to describe the hom prefunctors described above.
$$
\Precat(C,Z^Y)\cong \Prefun(Z,\homr(C,Y) \cong \Prefun(Z \tens C, Y)
$$
In a similar fashion to the context of precategories, we define equivariant hom objects as the equalizer of a pair of maps encoding the commutativity of homomorphisms and module structures. Given $Y$ an external right $C$-module, and $M$ an internal right $C$-module, we define
\begin{align*}
    \hom_C(M,Y) &= \mathrm{eq} \left(\homr(M,Y) \rightrightarrows \homr(M\circ C, Y)\right)
\end{align*}
and similarly for left $C$-modules. In the remainder of this section we will make our notation less cluttered by writing $\hom_C(M,Y)$ regardless of whether we are dealing with left or right modules. One should always keep in mind that the definition of a hom-object of left modules uses $\hom^l$ while for right modules it uses $\hom^r$. In addition, we will write $\Hom(-,-): \Precat^{op} \times \Precat \rightarrow \Precat$ for the internal homomorphism precategory bifunctor introduced in section 2, in order to avoid confusion with the external hom objects.

\begin{remark}
It is useful to observe that given $M$ an internal left $C$-module and $Y$ an external left $C$-module, the hom object $\hom_C(M,Y)$ is nothing more than the coproduct $\bigvee_{i\in S} \Nat_C(M(-,i),Y)$. Similarly, given $M$ and $N$ two internal right $R$-modules, we have
$$
\Hom_C(M,N) = \bigvee_{j\in S} \hom_C(M, N(-,j)) = \Nat_C(M(-,i), N(-,j))
$$
\end{remark}




\subsection{Modules over a category and its Koszul dual}

We now recall some of the notions of  \cite{DGI} necessary to formulate a translation of Theorem 4.9 of that paper to our context, relating $K(C)$-modules and $R$-cellular $C$-modules by a Quillen equivalence under mild assumptions. The material presented in this section is not really new, as its main result, Corollary \ref{cor:QuillenEq1}, follows from Theorem 3.9.3 of \cite{SchSh03}. From now on, in this section, the notation $\hom_C$ should be understood as {\it derived} homomorphisms of $C$-modules, i.e. $C$-equivariant homomorphisms from a cofibrant replacement to a fibrant replacement.

\begin{definition}\label{def:SmallObject}
An internal (left or right) $C$-module $X$ is said to be {\it small} if the representable functor $\hom_C(X,-)$ commutes with homotopy coproducts. As mentioned in (\cite{DGI}, 4.7) this is equivalent to saying that $X$ is finitely built from $C$, meaning that $X$ can be constructed from $C$ using equivalences, homotopy cofiber sequences and retracts.
\end{definition}

\begin{definition}
An {\it $R$-equivalence} of (left or right) $C$-modules $U\rightarrow V$ is a $C$-module map inducing an equivalence $\hom_C(R,U)\rightarrow \hom_C(R,V)$. If for an internal $R$-module $M$, the functor $\hom_C(M,-)$ sends $R$-equivalences to equivalences, $M$ is said to be {\it cellular}.
\end{definition}

\begin{definition}
An external right $C$-module $M$ is said to be {\it effectively constructible from $R$} if the evaluation map
$$
\hom_C(R,M)\tens_{K(C)} R \rightarrow M
$$
is an equivalence. Note that in the analogous definition for a left $C$-module $M$, the evaluation map has the order of the factors of the source swapped:
$$
R\tens_{K(C)} \hom_C(R,M)\rightarrow M
$$
\end{definition}

We will now use the results from section 3.9 of \cite{SchSh03} to derive an interesting corollary in our setting. First, looking at 3.9.1 of this same source, they define an endomorphism category construction $\E(\G)$ for a set $\G$ of objects in a spectral category $\D$. In our case, with $\D$ the category of $R$-cellular left $C$-modules, if we take $\G$ to be the set of representable functors $R(-,i)$ for $i\in S$, this definition coincides with the left-$C$-modules version of the Koszul dual $K(C)$, namely
$$
\E(\G)(R(-,i), R(-,j)) = \Nat_{C^{op}}(R(-,i), R(-,j)) = K(C)(i,j)
$$
Similarly it is also true that what they call $\Hom(\G, Y)$, for $Y$ an external left $C$-module, is the same as the definition of $\hom_C(R,Y)$ that we gave above. Now we want to assume that $\G$ is a set of compact generators for $\D$ (where compact means small as in definition \ref{def:SmallObject}). It is reasonable to make this assumption; think for example about the case $R=\Sigma$, or more generally the case where $R$ consists of a tower of monoids, each of which is compact/small in the base category $\V$. In addition, observe that every element of $\G$ is small if and only if the internal left $C$-module $R$ is small. Note however that every element of $\G$ being small does not imply that the {\it external} left $C$-module $\bigvee_{i\in S} R(-,i) = R\tens S$ is small.

\begin{corollary}\label{cor:QuillenEq1}
Assume that $R$ is small as a $C$-module. Then the adjunction
$$
R\tens_{K(C)} - : {_{K(C)}}\Mod \leftrightarrows {_C}\Mod : \hom_C(R,-)
$$
induces a Quillen equivalence between the category of left $K(C)$-modules and the category of $R$-cellular left $C$-modules.
\end{corollary}

\begin{proof}
The only thing to argue is that a $C$-module is $R$-cellular if and only if it is built from $R$, that is, if and only if it belongs to the smallest localizing subcategory of $\CMod$ that contains $R\tens S$. This fact is mentioned in 4.1 of \cite{DGI}, which refers to Theorem 5.1.5 of Hirschorn \cite{hirschhorn} for a proof. The reader should mind that, unfortunately, the terminology of these two sources differ. In \cite{hirschhorn}, the term {\it cellular} actually corresponds to the notion of small/finitely built presented in \cite{DGI} that we have recalled above.
\end{proof}

\subsection{The double dual of an external module}

For any left $C$-module $X$ there is a map
$$
\phi_X: X\rightarrow \hom_{K(C)}(\hom_C(X,\bar R),\bar R)
$$
which is the unit of the adjunction
$$
\hom_C(-,\bar R): \CMod \leftrightarrows \Mod_{_{K(C)}}^{op} : \hom_{K(C)}(-,\bar R)
$$
where $\bar R$ denotes the bar construction $B(C,C,R)$. Let us emphasize that $\hom_C(X,\bar R)$ is a {\it right} $K(C)$-module this time. Note that if we considered a right $C$-module instead, the source and target would look the same, although of course the hom-objects would be different (left hom instead of right hom and vice-versa).

We wish to generalize the observation made in the proof of Proposition 4.20 of \cite{DGI}, that for all $C$-modules finitely built from $R$ this map is an equivalence. First, let us look at the case $X=\bar R\tens S= \bigvee_{i\in S} \bar R(-,i)$. Then
\begin{align*}
    \hom_{K(C)}(\hom_C(X,\bar R),\bar R) &\cong \hom_{K(C)}(\Pi_j\hom_C(\bar R(-,j), \bar R), \bar R) \\
    &= \bigvee_i \Nat_{K(C)^{op}}(\Pi_j\Nat_{C^{op}}(\bar R(-,j), \bar R(-,i)), \bar R(-,i)) \\
    &\cong \bigvee_i \bar R(-,i)
\end{align*}
and one can check that the $k$-th component of $\phi_{R\tens S}$ is simply the natural inclusion of the $k$-th summand, which means the map from $R\tens S$ to its double dual is, up to isomorphism, just $\id_{R\tens S}$.

\begin{proposition}
The map $\phi_X$ is a weak equivalence for all $C$-modules finitely built from $R\tens S$. That is, for all $X$ that belong to the smallest thick category (i.e. closed under equivalences, cofiber/fiber sequences and retracts) that contains $R\tens S$.
\end{proposition}

\begin{proof}
The category of left $C$-modules $X$ such that $\phi_X$ is a weak equivalence is thick. It is closed under weak equivalences because the hom objects are derived and because of the 2-out-of-3 axiom, closed under cofiber/fiber sequences because $\CMod$ is stable, and closed under retracts because weak equivalences are closed under retracts.
\end{proof}

More concisely, this follows from the fact that the two functors $\hom_C(-,\bar R)$ and $\hom_{K(C)}(-,\bar R)$ preserve weak equivalences, cofiber/fiber sequences and retractions, which also yields the following result.

\begin{corollary}\label{prop:DoubleDualExternalModule}
Let us denote by $\Thick(X)$ the smallest thick subcategory generated by an object $X$ (be it a $C$-module or a $K(C)$-module). There is a Quillen equivalence
$$
\hom_C(-,\bar R): \Thick(R\tens S) \leftrightarrows \Thick(S\tens K(C))^{op} : \hom_{K(C)}(-,\bar R)
$$
between the category of left $C$-modules finitely built from $R\tens S$ and the opposite of the category of right $K(C)$-modules finitely built from $S\tens K(C)$.
\end{corollary}

\section{Double duality of modules and categories}

In this section, we focus on two basic types of $R$-algebras: free algebras and square zero extensions. We first explain why the natural algebra map from a square zero extension into its double dual is a weak equivalence. Then, we show that the Koszul dual of a square zero extension is weak equivalent, as an algebra, to a free algebra. In addition, any free algebra is equivalent to the Koszul dual of a square zero extension. As before, we tacitly assume that $R$ is levelwise cofibrant and that all algebras considered are cofibrant as $R$-bimodules. 

\subsection{The double dual of a square zero extension}

\begin{definition}
An $R$-algebra is said to be {\it finitely generated free} if it is of the form $\F(M)$ where $\F$ denotes the free $R$-algebra functor from section 2.2, and $M$ is finitely built from $R$ as an $R$-bimodule. The dual notion is that of a {\it trivial} $R$-algebra, also called a {\it square zero extension}, consisting of a coproduct $R\vee M$, which is said to be finite if $M$ is finitely built from $R$ in $_R\Mod_R$. The unit map is given by the canonical inclusion map, and the multiplication map puts together the algebra structure on $R$, the left and right $R$-module structures on $M$, and the trivial map $M\circ M \rightarrow * \rightarrow R$. In classical algebra, it is analogous to taking the quotient of the algebra of polynomials $R[M]$ by the ideal of elements of degree greater than or equal to $2$. Note that both of these kinds of algebra come naturally equipped with an augmentation map.
\end{definition}

The discussion on the double dual of an external module from section 4 is also valid for internal modules, namely we get a result analogous to Proposition \ref{prop:DoubleDualExternalModule}.


\begin{proposition}\label{prop:DoubleDualInternalModule}
There is a Quillen equivalence
$$
\R\Hom_C(-,\bar R): \Thick(R) \leftrightarrows \Thick(K(C))^{op}: \R\Hom_{K(C)}(-,\bar R)
$$
between the category of internal right $C$-modules finitely built from $R$ and the opposite of the category of left $K(C)$-modules finitely built from $K(C)$.
\end{proposition}

The reason we are interested in this internal modules version of the statement is that we can apply it to $C$ as a $C$-module in particular. Now, in general one cannot expect $C$ to be finitely built from $R$ as a $C$-module, as that is quite a strong condition and implies that the multiplicative structure on $C$ is not very interesting. However, a finite square zero extension $R\vee M$ is always finitely built from $R$ as an $R\vee M$-module, because of the simple observation that it is the coproduct of $R$ and $M$ as $R\vee M$-modules.

\begin{corollary}
For $R\vee M$ a finite square zero extension, the natural map
$$
R\vee M\rightarrow K(K(R\vee M))
$$
is a weak equivalence of $R$-algebras.
\end{corollary}

\begin{proof}
Let $C=R\vee M$. Since $C$ is cofibrant and finitely built from $R$ as a right $C$-module, and since we are using the category of EKMM spectra (in which all objects are fibrant) as the ground category $\V$, the unit map
$$
C \rightarrow \R\Hom_{K(C)}(\R\Hom_C(C,\bar R), \bar R) \cong \R\Hom_{K(C)}(\bar R,
\bar R)
$$
is a weak equivalence by Proposition \ref{prop:DoubleDualInternalModule}.
\end{proof}

We now want to show that this is true as well for finitely generated free $R$-algebras. As mentioned above, this is the natural dual notion, so one might expect that finite square zero extensions and finitely generated free algebras should be Koszul dual to one another. We will show this and use it to deduce that the natural map into the double dual is a weak equivalence of $R$-algebras in this case too.

\subsection{The dual of a square zero extension is free}

Throughout this discussion $C=R\vee M$ is a square zero extension, with $M$ an $R$-bimodule finitely built from $R$. We assume that $M$ is strictly directed and levelwise finite, where by {\it strictly directed} we mean either $M(i,j)= *$ for all $i\le j$ or $M(i,j)= *$ for all $i \ge j$. We also assume that $R$ is diagonal ($R(i,j)=*$ for all $i\neq j$).



\begin{proposition}\label{Prop:SquareZeroBarDescription}
The geometric realization $B(R,C,C)$ of the simplicial bar construction associated to $(R,C,C)$ is isomorphic to the coproduct
$$
R\vee (I\Smash M) \vee (S^1 \Smash I \Smash M^{\circr 2}) \vee \cdots \vee (({S^1})^{\Smash n-1} \Smash I \Smash M^{\circr n}) \vee \cdots
$$
where $I$ denotes the interval $[0,1]$ with $0$ as base point.
\end{proposition}

\begin{proof}
The first thing to notice is that since the simplices of $B(R,C,C)_\bullet$ are chains of composable elements of $R$ and $M$, we can introduce a weight grading, elements of $R$ having weight $0$ and elements of $M$ having weight $1$. With this in mind, the weight $n$ component $B_{(n)}(R,C,C)_\bullet$ is defined as follows. For $k$ a natural number (the simplicial index), $B_{(n)}(R,C,C)_k$ is the coproduct of all $R\circr X_1 \circr \cdots \circr X_{k+1}$, where among $X_1,\ldots, X_{k+1}$ there are $n$ copies of $M$ and the other $X_i$'s are copies of $R$.
Because the faces and degeneracies of the bar construction preserve the weight grading (the basepoint being in every weight degree), this actually defines a "sub-simplicial object" and we can write
$$
B(R,C,C)_\bullet = \bigvee_{n\ge 0} B_{(n)}(R,C,C)_\bullet
$$

It is then enough to show that the geometric realization $B_{(n)}(R,C,C)$ of the $n$-th weight degree is isomorphic to $({S^1})^{\Smash n-1} \Smash I \Smash M^{\circr n}$.

First, recall the definition of the degeneracy maps of the simplicial bar construction: they consist in inserting a copy of $R$ between the leftmost $R$ and the rightmost $C$ of $B(R,C,C)_m=R\circr C \circr \cdots \circr C$. This means that the only non-degenerate simplices of $B_{(n)}(R,C,C)$ lie in the $M^{\circr n}$ component in simplicial degree $n$ and the \linebreak $M^{\circr n}\circr R$ component in simplicial degree $n+1$. From this we can see that $B_{(n)}(R,C,C)$ is isomorphic to the geometric realization of $X_\bullet \Smash M^{\circr n}$ where $X_\bullet$ is the based simplicial set with one non-degenerate $n$-simplex $\beta_n$, one non-degenerate $n+1$-simplex $\alpha_n$, and face maps $d_i: X_{n+1} \rightarrow X_n$ satisfying $d_{n+1}(\alpha_n)=\beta_n$ and $d_i(\alpha_n)=*$ for $i\neq n+1$.

We can write $X_\bullet$ as the join of simplicial sets $S^0 \star \cdots \star S^0 \star \Delta^1$, where by join we mean that $Y \star Z$ is the pushout
$$
\xymatrix{
& Y \smash Z \smash 1 \ar[d] \ar[r]& Z \ar[dd]\\
Y \smash Z \smash 0 \ar[r] \ar[d] & Y \smash Z \smash \Delta^1 & \\
Y \ar[rr] & & Y \star Z
}
$$
This definition is a pointed version of the one presented in Section 4.2.1 of \cite{lurieHigherTopos} that is due to Joyal. As mentioned in Proposition 4.2.1.2 of this source, it is equivalent to the more traditional definition of join of simplicial sets that involves the Day convolution product of augmented simplicial sets. By applying the geometric realization functor to this diagram, we see that the realization of the join is the (topological) join of the realizations. In particular, we get a natural homeomorphism $\vert S^0 \star -\vert \cong S^1 \smash \vert-\vert$ and therefore  the geometric realization of $X_\bullet$ can be written as $(S^1)^{\smash n-1}\smash I$, as desired. 
\vspace{5mm}
\end{proof}

Under this description, the right action of $C=R\vee M$ is given by the canonical right action of $R$ on each weight component and an action of $M$ seen as a trivial monoid without unit. Given a weight component $X_\bullet \Smash M^{\circr n}$ of the simplicial bar construction $B(R,C,C)_\bullet$ (see the proof of Proposition \ref{Prop:SquareZeroBarDescription} above), this $M$-action is given by combining the concatenation $M^{\circr n}\circr M\rightarrow M^{\circr n+1}$ and the map
$$
(S^0)^{\star n-1}\star \Delta^1 \rightarrow (S^0)^{\star n} \rightarrow (S^0)^{\star n} \star \Delta^1
$$
obtained by composing the natural quotient and inclusion maps.

\begin{proposition}
The geometric realization $B(R,C,R)$ of the simplicial bar construction associated to $(R,C,R)$ can be described as a coproduct
$$
\bigvee_{n\ge 0} (\Sigma M)^n
$$
where $\Sigma M$ is the levelwise suspension of $M$.
\end{proposition}

\begin{proof}
This is very analogous to the proof of Proposition \ref{Prop:SquareZeroBarDescription}, with the difference that in this case {\it all} face maps $d_i$ send $\alpha_n$ to the basepoint.
\end{proof}

We will build a weak equivalence of algebras $\vp: \F(D(\Sigma M)) \rightarrow K(C)$ where $D(-) = \hom_R(-, R)$ is the dual functor for right $R$-modules. To get a map of algebras, we just need a map of $R$-bimodules $D(\Sigma M) \rightarrow K(C)$ or, in other words, an evaluation map $D(\Sigma M) \circr B(R,C,C) \rightarrow B(R,C,C)$ that is a map of $(R,C)$-bimodules. We can use the description obtained above to make this definition easier. First we do this for weight 0, by considering the trivial map
$$
D(\Sigma M) \circr R \rightarrow * \rightarrow B(R,C,C)
$$
Then, in weight 1, we have
$$
D(\Sigma M) \circr (I \Smash M) \rightarrow D(\Sigma M) \circr \Sigma M \rightarrow R
$$
Finally, for weight $n\ge 2$,
$$
\xymatrix{
D(\Sigma M) \circr (({S^1})^{\Smash n-1} \Smash I \Smash M^{\circr n}) \ar[d] \\
D(\Sigma M) \circr \Sigma M \circr (({S^1})^{\Smash n-2} \Smash I \Smash M^{\circr n-1}) \ar[d] \\
({S^1})^{\Smash n-2} \Smash I \Smash M^{\circr n-1}
}
$$
gives the desired evaluation map. Notice that this evaluation map preserves weight (considering that elements of $D(\Sigma M)$ have weight -1).

We obtain a map of algebras $\vp: \F(D(\Sigma M)) \rightarrow K(C)$. We will show that the composite
$$
\F(D(\Sigma M)) \xrightarrow{\vp} K(C) \xrightarrow{\psi} \Hom_R(B(R,C,R),R)
$$
is a weak equivalence, and hence $\vp$ is one as well (here $\psi$ is the weak equivalence induced from the map $B(R,C,C)\rightarrow R$ in Proposition \ref{prop:DescriptionDual}). The first step is to show that it is the same as the canonical map
$$
\bigvee_{n\ge 0} D(\Sigma M)^n \rightarrow \prod_{n\ge 0} D((\Sigma M)^n)
$$
For $n,k$ two natural numbers, let us look at the component of $\varrho =\psi\vp$ from weight $n$ to $k$, which we will write $\varrho_{n,k}: D(\Sigma M)^n \rightarrow D((\Sigma M)^k)$. Since the augmentation $\epsilon: B(R,C,C) \rightarrow R$ sends all elements of non-zero weight to the basepoint, and the image of $D(\Sigma M)^n$ under $\vp$ consists of homomorphisms that decrease weight by $n$ and are trivial on components of weight less than $n$, we see that $\varrho_{n,k}$ is trivial if $n\neq k$.

Let us look at the map
$$
D(\Sigma M)^n \circ_R B(R,C,C) \rightarrow R
$$
that is adjoint to $\epsilon_* \vp$. From what we just said, we know the only non-trivial component of it is
$$
D(\Sigma M)^n \circ_R (S^1)^{n-1} \smash I \smash M^n \rightarrow R
$$
The adjoint map to $\varrho_{n,n}$ is then the factorization of this right $C$-module map through $D(\Sigma M)^m \circ_R (S^1)^{n}\smash M^n$, which by definition of $\vp$ corresponds to the canonical map consisting in successive applications of the evaluation map. 

We will now show that $\varrho_{n,n}$ is a weak equivalence for all $n$, by proving the more general result that for $X$ and $Y$ two directed $R$-bimodules finitely built from $R$, the map
$D(X) \circ_R D(Y) \rightarrow D(Y\circ_R X)$
is a weak equivalence. Recall we assume R is diagonal and we write $R_i$ for $R(i,i)$ (i.e. $R$ is just a tower of monoids in $\V$). Let us look at this map in level $(i,j)$. On one hand
$$
D(X) \circ_R D(Y) (i,j) = \bigvee_{k} \map_{R_i}(X(k,i), R_i) \smash_{R_k} \map_{R_k}(Y(j,k),R_k)
$$
and on the other hand
$$
D(Y\circ_R X)(i,j) = \map_{R_i}(\vee_k Y(j,k) \smash_{R_k} X(k,i), R_i)
$$
The wedge sums involved are finite, which means that
$$
\vee_k \map_{R_i}( Y(j,k) \smash_{R_k} X(k,i), R_i) \rightarrow D(Y\circ_R X)(i,j)
$$
is a weak equivalence, and we are thus left with proving that
$$
\xymatrix{
\map_{R_i}(X(k,i), R_i) \smash_{R_k} \map_{R_k}(Y(j,k),R_k) \ar[d]\\
\map_{R_i}(Y(j,k) \smash_{R_k} X(k,i), R_i)
}
$$
is one as well. For $Y=R$ this is trivial; we can conclude by recalling that $Y(j,k)$ is finitely built from $R_k$ by assumption and that both of the functors involved in the source and target of the latter map preserve triangles (i.e. cofiber/fiber sequences), weak equivalences and retracts, yielding the following result.

\begin{proposition}
The map $\varrho_{n,n}$ is a weak equivalence for all $n$.
\end{proposition}

Now let $i,j\in \N$ and consider the $(i,j)$-th level of $\varrho$ 
$$
\bigvee_{n\ge 0} D(\Sigma M)^n(i,j) \rightarrow \prod_{n\ge 0} D((\Sigma M)^n)(i,j)
$$
Since we assumed $M$ is directed and that $\vee_i M(i,i)=*$, there are only a finite number of non-trivial weight components in both its source and target, so the wedge and products are actually finite.

\begin{corollary}\label{Corollary:DualTrivialIsFree}
Let $M$ be an $R$-bimodule and $C= R\vee M$ the corresponding square zero extension. Then the map
$$
\varrho: \bbF (D(\Sigma M)) \xrightarrow{\varphi} K(C) \xrightarrow{\psi} \Hom_R(B(R,C,R),R)
$$
is a weak equivalence of $R$-bimodules, meaning $\vp$ is a weak equivalence of\linebreak $R$-algebras.
\end{corollary}


\begin{remark}
Notice that this also shows every free $R$-algebra is the Koszul dual of some square zero extension, taking. Indeed, for any free algebra $\bbF (N)$, we can just take $M=\Omega D(N)$ and apply Corollary \ref{Corollary:DualTrivialIsFree}.
\end{remark}



\bibliography{Sources}{}
\bibliographystyle{amsalpha}

\noindent
\Address

\end{document}